\documentclass[12pt,a4paper]{amsart}

\makeatletter
\usepackage{enumitem}
\usepackage[utf8]{inputenc}
\usepackage[english]{babel}
\usepackage{amsmath}
\usepackage[usenames,dvipsnames]{color}
\usepackage[all,cmtip]{xy}
\usepackage{fullpage}
\usepackage{hyperref}
\usepackage{amsthm,amscd,amssymb, color, amsfonts}
\usepackage{fullpage}
\usepackage{bm}
\usepackage{latexsym,mathtools}
\usepackage{graphicx}

\theoremstyle{plain}
\newtheorem{theorem}{Theorem}
\newtheorem{proposition}[theorem]{Proposition}
\newtheorem{lemma}[theorem]{Lemma}
\newtheorem{corollary}[theorem]{Corollary}

\theoremstyle{definition}
\newtheorem{definition}[theorem]{Definition}

\theoremstyle{remark}
\newtheorem{remark}[theorem]{Remark}

\DeclareMathOperator{\hyp}{hyp}

\let\Im\relax
\DeclareMathOperator{\Im}{Im}
\let\Re\relax
\DeclareMathOperator{\Re}{Re}
\definecolor{green}{rgb}{.4,.6,.4}

\def\SL {\text{\rm SL}}
\def\PSL {\text{\rm PSL}}
\def\CC {{\mathbb C}}
\def\RR {{\mathbb R}}
\def\ZZ {{\mathbb Z}}
\def\HH {{\mathbb H}}
\def\NN {{\mathbb N}}
\def\G {\Gamma}
\def\vp {\varphi}

\begin{document}

\title{Construction of Poincar\'e-type series by generating kernels}

 \author{ Yasemin Kara, Moni Kumari, Jolanta Marzec, Kathrin Maurischat, Andreea Mocanu \and Lejla Smajlovi\'{c}}

\keywords{Wave  distribution, automorphic kernel, real weight Laplacian, multiplier system, sup-norm bound}

\subjclass[2010]{11F72, 43A22, 58J35, 35J05}







\address{\rm {\bf Yasemin Kara},
Faculty of Arts and Sciences, Mathematics Department,
Bogazici University, Bebek, Istanbul, 34342, Turkey}
\email{\sf yasemin.kara@boun.edu.tr}
 \address{\rm {\bf Moni Kumari}, School of Mathematics, Tata Institute of Fundamental Research, India}
\email{\sf moni@math.tifr.res.in}
\address{\rm {\bf Jolanta Marzec}, Department of Mathematics, Technical University of Darmstadt, Schloss\-gar\-ten\-stra\-ße~7,
64289 Darmstadt, Germany}
\email{\sf marzec@mathematik.tu-darmstadt.de}
 \address{\rm {\bf Kathrin Maurischat}, Lehrstuhl A Mathematik, RWTH Aachen University, Templergraben~55, 52062 Aachen, Germany}
\curraddr{}
\email{\sf kathrin.maurischat@matha.rwth-aachen.de}
\address{\rm {\bf Andreea Mocanu} }
\email{\sf andreea.mocanu@protonmail.com}
\address{\rm {\bf Lejla Smajlovi\'{c}}, Department of Mathematics,
University of Sarajevo,
Zmaja od Bosne 35,
71000 Sarajevo,
Bosnia and Herzegovina}
\email{\sf lejla.smajlovic@efsa.unsa.ba; lejlas@pmf.unsa.ba}
\maketitle

\begin{abstract}
Let $\Gamma\subset\PSL_2(\RR)$ be a Fuchsian group of the first kind having a fundamental domain with a finite hyperbolic area, and let $\widetilde\Gamma$ be its cover in $\SL_2(\RR)$. Consider the space of twice continuously differentiable, square-integrable functions on the hyperbolic upper half-plane, which transform in a suitable way with respect to a multiplier system of weight $k\in\RR$ under the action of $\widetilde\Gamma$. The space of such functions admits the action of the hyperbolic Laplacian $\Delta_k$ of weight $k$. Following an approach of \cite{JvPS16} (where $k=0$), we use the spectral expansion associated to $\Delta_k$ to construct a wave distribution and then identify the conditions on its test functions under which it represents automorphic kernels and further gives rise to Poincar\'e-type series. An advantage of this method is that the resulting  series may be naturally  meromorphically continued to the whole complex plane. Additionally, we derive sup-norm bounds for the eigenfunctions  in the discrete spectrum of $\Delta_k$.
\end{abstract}

\tableofcontents

\section{Introduction}

Let $\Gamma\subset\mathrm{PSL}_{2}(\mathbb{R})$ be a Fuchsian group of the first kind having a fundamental domain  $\mathcal{F}$  with a finite hyperbolic area.  It acts on the complex upper half-plane $\mathbb H$ and the quotient space can be identified with the Riemann surface $M=\Gamma\backslash \mathbb H$.

We fix a real weight $k$, such that there exists a unitary multiplier system $\chi$ of weight $k$ on the cover $\widetilde{\Gamma}$ of $\Gamma$ in $\mathrm{SL}_{2}(\mathbb{R})$. Throughout the paper we will assume the weight $k$ and the multiplier system to be arbitrary but fixed.

The hyperbolic Laplacian of weight $k$ on $M$ is the operator
$$
\Delta_k\:=\:- y^2(\partial_x^2+\partial_y^2)+2kiy\partial_x\:
$$
acting on the space $D_k$ of all twice continuously differentiable, square-integrable functions on $\mathbb{H}$ which transform in a suitable way with respect to the weight $k$ unitary multiplier system  $\chi$ on $\widetilde{\Gamma}$ (a precise definition is given in Section \ref{sec: prelim}).

The operator $\Delta_k$ is a special case of a differential operator investigated by Maass in \cite{M52} (see also \cite{roelcke1}); in some papers (e.g. \cite{Oshima}), $\Delta_k$ is referred to as the Maass-Laplacian.  It is the analogue of the non-Eucledian Laplacian for non-analytic automorphic forms on  $\Gamma\backslash \mathbb H$ of weight $k$.
Namely, the weighted Laplacian preserves the transformation behavior of functions from $D_k$ (cf. formula \eqref{delta property} below); it can be represented as a composition of differential operators of the first order (the lowering and raising Maass operators) mapping weight $k$ forms into forms of weight $k\pm 1$ (or $k\pm 2$, depending on the scaling taken); the kernel of $\Delta_k - k(1-k)$  is isomorphic to the set of meromorphic differentials on $\mathcal{F}$ of weight $k$ with unitary multiplier system, see \cite[Section 1.4]{Fi87} and references therein.

The spectral theory of $\Delta_k$ was carefully developed in \cite{roelcke2, roelcke3}, \cite{El73a, El73b} and \cite{Pa75, Pa76a,Pa76b}. It was proved that the Hilbert space of all square-integrable functions on $\mathbb{H}$ which transform in a suitable way with respect to the weight $k$ unitary multiplier system on $\widetilde{\Gamma}$
is a direct sum of countably many finite dimensional eigenspaces spanned by the eigenfunctions $\varphi_j$ associated to the discrete eigenvalues $\lambda_j$ of $\Delta_k$ and, in the case when the surface is non-compact, the Eisenstein series associated to each cusp and corresponding to the continuous spectrum.

In the seminal paper \cite{Fay77}, Fay computed the basic eigenfunction expansions (in both rectangular and hyperbolic polar coordinates) of the resolvent kernel for the operator $\Delta_k$ acting on the Hilbert space of real weight $k$ automorphic forms. He applied it to the construction of an automorphic ``prime form'' and automorphic functions with prescribed singularities. Following the ideas of Selberg \cite{Sel56}, Hejhal developed the trace formula for $\Delta_k$ on the space $\Gamma\backslash \mathbb H$ and derived various applications of it, such as a  distribution of pseudo-primes and the computation of the dimension of the space of classical cusp forms of weight $2k\in\mathbb{N}$.

\subsection{Poincar\'e-type series}
The resolvent kernel for $\Delta_k$ is the special case of Poincar\'e-type series, which is the series defined by summing over the group $\Gamma$ (or over its cover $\widetilde{\Gamma}$) the weight $k$ point-pair invariant, multiplied by certain other factors depending on the element of $\Gamma$ (or $\widetilde{\Gamma}$) in the index of the sum.
Loosely speaking, the weight $k$ point-pair invariant is a function $k(z,w)$ of $z,w \in\mathbb{H}$ which is radial/spherical (meaning it depends only upon the hyperbolic distance between $z$ and $w$) and which transforms ``nicely'' with respect to the weight $k$ pseudo-action of $\widetilde{\Gamma}$.
In most applications (and so is the case with the resolvent kernel), the point-pair invariant is taken to depend upon one or two complex parameter(s) with large enough real part (to ensure the convergence of the series).

When the point-pair invariant is well-chosen, Poincar\'e-type series become very important objects of study. This is because their Fourier expansions (in different coordinates) together with a Laurent/Taylor series expansion in an additional complex variable $s$ (usually at $s=1$ or $s=0$) carry important information. For example, when the multiplier system is identity, the constant term in the Laurent series expansion of the resolvent kernel at $s=1$ gives rise to a holomorphic automorphic ``prime form'', see \cite[Theorem 2.3]{Fay77}. Such a form is an important object in the construction of automorphic forms with prescribed singularities. 
In case when $M$ is compact and the multiplier system is one-dimensional, the constant term in the Laurent series expansion of the resolvent kernel at $s=1$ gives rise to a unique Prym differential with multipliers, see \cite[p. 163]{Fay77}.
The resolvent kernel asymptotic is used more recently in \cite{FJK19} with $k\in\tfrac{1}{2}\mathbb{N}$ to establish effective sup-norm bounds on average for weight $2k$ cusp forms for $\Gamma$.

In the series of papers \cite{Br85, Br86a, Br86b}, Bruggeman constructed families of Eisenstein and Poincar\'e series on the full modular group, which depend on a complex variable $s$ parameterizing the eigenvalue $1/4-s^2$ of $\Delta_k$, and also depend continuously (in fact real-analytically) on the weight $k$ (which in this special case can be taken to belong to the interval $(0,12)$). 
Bruggeman proved that, for $k\neq 0$, all square-integrable modular forms (of a certain type) occur in such families.

\subsection{Our results}
In this paper, we follow an approach of \cite{JvPS16} and \cite{CJS18} towards the construction of Poincar\'e-type series associated to the weighted Laplacian.
It is similar to Bruggeman's in the sense that we undertake the ``operator'' approach, looking at appropriate distributions. However, our starting point is the ``wave distribution'' (see Definition \ref{def:wave_dist}), a concept which does not appear in \cite{Br85}--\cite{Br86b}.

We illustrate this approach in Section \ref{sec:kernel construction}, by constructing a new Poincar\'e-type series $K_s(z,w)$, for $z,w\in \mathcal F$ and $\mathrm{Re}(s) \gg0$, which transforms ``nicely'' with respect to the weight $k$ multiplier system.
We then obtain its meromorphic continuation with respect to the $s$ variable and deduce its representation in terms of the sum over the group of a certain point-pair invariant.

More precisely, starting with the spectral expansion theorem (\cite[Theorem 1.6.4]{Fi87}), we construct the wave distribution associated to the weighted Laplacian. We prove in Proposition \ref{prop:wave-dist} that the wave distribution acts on a rather large space of test functions and that it can be represented as an integral operator with a certain kernel (Theorem \ref{thm:integral kernel}). In Proposition \ref{prop:wave-dist}, we also derive sufficient conditions on the test function so that the wave distribution acting on this test function produces an $L^2$-automorphic kernel.

To guarantee the absolute convergence of the aforementioned objects we need bounds for their discrete and, in case when the surface is non-compact, continuous spectra. It turned out that even though the spectral properties of $\Delta_k$ are well studied, both
analytically and computationally (see \cite{St08}), the properties of the eigenvalues associated
to its discrete spectrum that are different from the minimal eigenvalue $|k|(1-|k|)$ did not get much attention in the non-compact
setting. This is probably because in the non-compact setting the discrete spectrum still remains very mysterious; for example it is not even known in general whether it is finite or infinite.

For that reason, in Section \ref{sec:sup norm} we prove the
sup-norm bound $\sup_{z\in\mathcal{F}}|\varphi_j(z)| \ll |\lambda_j|$  for the eigenfunctions associated to discrete eigenvalues $\lambda_j$ of $\Delta_k$, uniform in $j$. This result is of independent interest, because it is proved in a general setting of a possibly non-compact surface, real weight $k$ and vector-valued
eigenfunctions $\varphi_j$. In the non-compact setting, we also derive the sup-norm bound for the growth of a certain weighted integral of the (vector-valued) Eisenstein series along the critical line (Proposition
\ref{prop:sup_norm_bound}.(b)).

The definition of the wave distribution enables one to construct automorphic kernels through the action of the wave distribution on suitably chosen test functions.
In Section \ref{sec:kernel construction}, a new $L^2$-automorphic kernel $K_s(z,w)$, called the basic automorphic kernel, is constructed through the action of the wave distribution on the test function $g_s(u)=\tfrac{\Gamma(s-1/2)}{\Gamma(s)} \cosh(u)^{-(s-1/2)}$ for $\mathrm{Re}s\gg0$. The kernel  $K_s(z,w)$ is called the basic kernel, because, as will be seen in \cite{KKMMMS20}, both the resolvent kernel and, consequently, the Eisenstein series, can be expressed in terms of this kernel and its translates in the $s$-variable. It is analogous to the basic automorphic kernel constructed in \cite{CJS18} in the setting of smooth, compact, projective K\"ahler varieties.

Using the properties of the wave distribution, it is possible to construct Poincar\'e-type series that are not square-integrable by taking appropriate sums/integrals of the wave distribution (see e.g. \cite[Section 7]{JvPS16} in the special case of the multiplier system equal to $1$). We leave this investigation to the subsequent paper \cite{KKMMMS20}.

This approach to the construction of Poincar\'e-type series has many advantages. Firstly, the construction depends only on the spectral properties of the Laplacian, and we believe it can be applied in more general settings, with the Laplacian replaced by the Casimir element (see a discussion in Section \ref{sec:weightedLap} below). Second, the problem of the meromorphic continuation of Poincar\'e series, which is usually attacked by means of Fourier series expansion and serious analytic considerations related to the coefficients in those series, is simplified. Namely, the meromorphic continuation essentially boils down to establishing a suitable functional relation between the Fourier transform of the test function at $s$ and at $s+\alpha$, for a suitable translation parameter $\alpha$ (see Lemma \ref{H-relation} below). For this reason, the scaling factor $\tfrac{\Gamma(s-1/2)}{\Gamma(s)}$ appears in the test function $g_s(u)$ above.

Moreover, this approach provides additional flexibility in the construction of Poincar\'e series, depending on the desired properties of the series, under the action of $\Delta_k$. Namely, assume that one is interested in the construction of Poincar\'e series $P_s(z,w)$ on $M$, such that $\Delta_k P_s(z,w)$ equals a certain function of $P_s(z,w)$. Then, representing $P_s(z,w)$ as the wave distribution acting on an unknown test function, this construction boils down to solving a second order differential equation satisfied by this test function, with some natural boundary conditions, such as e.g. decay to zero as $\mathrm{Re}(s)\to\infty$. This task is not easy, but it may turn out to be easier than solving the partial differential equation that is to be satisfied by the point-pair invariant generating the series $P_s(z,w)$ as a sum over the group $\Gamma$ (or its cover $\widetilde{\Gamma}$).

\subsection{Outline of the paper}

The paper is organized as follows: in Section \ref{sec: prelim}, we introduce the basic notation, define the weighted Laplacian, the unitary multiplier system and the spaces of functions we are interested in and we recall the spectral expansion theorem. In Section \ref{sec:aut kernel}, the construction of the geometric automorphic kernel is presented, following the approach undertaken in \cite{He76} and \cite{He83} and the pre-trace formula for the resolvent kernel is recalled from \cite{Fi87}. Section \ref{sec:sup norm} is devoted to proof of the non-trivial sup-norm bound for the eigenfunctions of the weighted Laplacian, a result necessary for the construction of the wave distribution associated to $\Delta_k$ in Section \ref{Sec:Wave dist}. Properties of the wave distribution are identified in Section \ref{Sec:Wave dist} and applied to the construction of the basic automorphic kernel in Section \ref{sec:kernel construction}.

\subsection{Acknowledgement} The authors would like to thank the organizers and sponsors of WINE3 for providing a stimulating atmosphere for collaborative work.
\newpage

\section{Preliminaries}\label{sec: prelim}

\subsection{Basic notation}\label{basic-notation}

Let $\Gamma\subset\mathrm{PSL}_{2}(\mathbb{R})$ denote a Fuchsian
group of the first kind. It acts by fractional linear transformations on the hyperbolic upper half-plane $\mathbb{H}:=\{x+iy\,|\,x,y\in\mathbb{R};\,y>0\}$.
We choose once and for all a connected fundamental domain $\mathcal{F}\subseteq \HH$ for $\Gamma$.
We further assume $\mathcal{F}$ (and therefore every fundamental domain) to have finite hyperbolic area.
Then  $M:=\Gamma\backslash\mathbb{H}$ is a finite volume hyperbolic Riemann surface, which we allow to have elliptic fixed points and $c_{\Gamma}$ cusps.
Locally,  $M$ is identified with its universal cover $\mathbb{H}$, and each point on $M$ has a unique representative in $\mathcal F$.
We rely on this identification of $M$ with $\mathcal F$ whenever a definition of a function on $M$ uses the choice of a representative in $\HH$. This in particular applies to the kernel functions in this paper.

Let $\widetilde{\Gamma}$ denote the cover of $\Gamma$ in $\mathrm{SL}_{2}(\mathbb{R})$, i.e. the set of all matrices $\gamma\in\mathrm{SL}_{2}(\mathbb{R})$ such that $[\pm \gamma] \in \Gamma$. Throughout this paper, assume that $\widetilde{\Gamma}$ contains $-I$, where $I$ stands for the identity element of $\mathrm{SL}_{2}(\mathbb{R})$.

Let $\mu_{\mathrm{hyp}}$ denote the hyperbolic metric on $M$, which is compatible with the complex structure of $M$, and has constant negative curvature equal to minus one.
The hyperbolic line element $ds^{2}_{\hyp}$ is given by
$
\displaystyle{ds^{2}_{\hyp}:=\frac{dx^{2}+dy^{2}}{y^{2}}}.
$
Denote the hyperbolic distance from $z\in\mathbb{H}$ to $w\in\mathbb{H}$ by $d_{\mathrm{hyp}}(z,w)$. It satisfies the relation
\begin{align}\label{rel_cosh_u}
\cosh\bigl(d_{\mathrm{hyp}}(z,w)\bigr)= 1+2 u(z,w),
\end{align}
where
\begin{align}
\label{def_u}
u(z,w):=\frac{\left|z-w\right|^{2}}{4\,\mathrm{Im}(z)\mathrm{Im}(w)}\,.
\end{align}
In the sequel, we will need the displacement function $\sigma(z,w)$, which is defined as
\begin{align}
\label{def_sigma}
\sigma(z,w):=1+\frac{\left|z-w\right|^{2}}{4\,\mathrm{Im}(z)\mathrm{Im}(w)}=
\frac{\left|z-\overline{w}\right|^{2}}{4\,\mathrm{Im}(z)\mathrm{Im}(w)}\,.
\end{align}

\subsection{Weighted Laplacian}\label{sec:weightedLap}
For any real $k$, denote by
$$
\Delta_k\:=\: -y^2(\partial_x^2+\partial_y^2)+2kiy\partial_x
$$
the hyperbolic Laplacian on $M$ of weight $k$, which will be applied to twice differentiable functions $f:\mathbb{H}\to\CC$.
H. Maass \cite{M52} introduced in broader generality, for real numbers $\alpha$ and $\beta$, the differential operator
$$
\Delta_{\alpha,\beta}\:=\: -y^2\left( \partial_x^2 + \partial_y^2 \right) + (\alpha - \beta)iy\partial_x - (\alpha +\beta)y\partial_y\:.
$$
Specializing to $\alpha+\beta=0$, we recover the classical Laplace-Beltrami operator on $\mathbb H$ of weight  $\alpha-\beta$ (which is, among others,
subject of Roelcke's work \cite{roelcke1,roelcke2,roelcke3}).

There is a slight ambiguity in the notation used: the  operator $\Delta_{\alpha,\beta}$ with $\alpha - \beta=\alpha +\beta=k$ is also called the weighted Laplacian of weight $k$ in the literature. It is that one which is used in the theory of mock modular forms (see e.g. \cite{bringmann-kudla}).

Our choice of the weighted Laplacian is the specialization to weight $2k$ of the Laplace-Beltrami operator
\begin{equation*}
 \widetilde \Delta\:=\:-y^2\left( \partial_x^2 + \partial_y^2 \right) + y\partial_x\partial_\theta\:,
\end{equation*}
which in turn equals the Casimir operator for $\mathrm{SL}(2,\mathbb{R})$, up to a multiplicative constant.
More precisely, the action of $\mathrm{SL}(2,\mathbb{R})$ on $L^2(\widetilde\Gamma\backslash \mathrm{SL}(2,\mathbb{R}))$ by right translations comes along with an action of its Lie algebra on $C^\infty$-vectors, given by differential operators.
The Casimir element generates the center of the universal enveloping Lie algebra and, written with respect to the coordinates $\partial_x$, $\partial_y$, $\partial_\theta$, this operator coincides with the Laplace-Beltrami operator $\tilde \Delta$ above. By Schur's lemma, the Casimir acts as a constant on any irreducible representation of $\mathrm{SL}(2,\mathbb{R})$. In turn, any eigenfunction of the Casimir, respectively the Laplace-Beltrami,
together with its $\mathrm{SL}(2,\mathbb{R})$-translates generates  an irreducible representation.
On the other hand, when restricting to eigenfunctions of weight $2k$ for the maximal compact subgroup $\mathrm{SO}(2)$ of $\mathrm{SL}(2,\mathbb{R})$, the Casimir operator specializes to our choice of the weighted Laplacian $\Delta_k$. In turn, the isomorphism of $\widetilde\Gamma\backslash\mathrm{SL}(2,\mathbb{R})/\mathrm{SO}(2)$ with $\Gamma\backslash \mathbb H$ induces an isomorphism of the $\mathrm{SO}(2)$-eigenfunctions on $\widetilde\Gamma\backslash\mathrm{SL}(2,\mathbb{R})$ with automorphic forms of weight $2k$ on $\mathbb H$.
\subsection{Unitary multiplier system}
For every $\gamma=\left(
                     \begin{array}{cc}
                       * & * \\
                       c & d \\
                     \end{array}
                   \right) \in\widetilde\Gamma$ and every complex number $z$,
define $j(\gamma,z):=cz+d$ and $J_{\gamma,k}(z):=\exp(2ik\arg j(\gamma,z))$.
\begin{definition}\label{D:point-pair}
 A function $\mu:\HH^2\to \mathbb C^\ast$ satisfying the transformation property
 $$
 \mu(\gamma z,\gamma w)=\mu(z,w)J_{\gamma,k}(z)J_{\gamma,k}(w)^{-1}
 $$
 for all $\gamma\in\mathrm{SL}_2(\RR)$ and all $z,w\in\HH$ is called a \textit{weight $k$ point-pair invariant}.
\end{definition}

Note that, due to the fact that $\mathrm{SL}_2(\RR)$ acts transitively on point-pairs of a fixed hyperbolic distance, a point-pair invariant of weight zero is just an ordinary point-pair invariant depending only on the hyperbolic distance of the point-pair $(z,w)$. Further,
if $\mu$ is a weight $k$ point-pair invariant and $\Phi$ is a weight zero point-pair invariant,
then $\mu\cdot\Phi$ is also a weight $k$ point-pair invariant. Furthermore, if $\nu$ is also a weight $k$ point-pair invariant, then $\mu/\nu$ is a point-pair invariant of weight zero.

\begin{lemma}\label{lemma-H_k-definition}
Decompose the real number $k=k_1+k_2$ with $k_1\in\mathbb Z$ and $k_2\in (-\frac{1}{2},\frac{1}{2}]$, and define $z^k=z^{k_1}\exp(k_2 \log z)$ for the principal branch of the complex logarithm $\log z$.
 The function $H_k:\HH^2\to\mathbb C^\ast$ given by
 $$
H_{k}(z,w):=\left(-\frac{\lvert z-\overline{w}\rvert^2}{(z-\overline{w})^2}\right)^{k}=
\overline{\left(\frac{z-\overline{w}}{w-\overline{z}}\right)}^k=\left(\frac{(1-\zeta)^2}{\lvert 1-\zeta\rvert ^2}\right)^k, \quad \mathrm{for} \quad \zeta=\frac{z-w}{z-\overline{w}}\:,
$$
is a weight $k$ point-pair invariant.
\end{lemma}
\begin{proof}
The function $r:\HH^2\to\mathbb C$ given by
$$
r(z,w)=1-\frac{z-w}{z-\overline{w}}=\frac{2i\Im w}{z-\overline{w}}
$$
transforms under $\mathrm{SL}_2(\mathbb R)$ as
\begin{equation}\label{r-transformation}
r(\gamma z, \gamma w)=r(z,w)\left(\frac{cz+d}{cw+d}\right).
\end{equation}
Note that, for all $z,w\in\mathbb H$, we have $z-\overline{w}\in\HH$. In particular, $0\leq\arg(z-\overline{w})<\pi$, 
which implies that
$$-\frac{\pi}{2} <\arg(r)=\frac{\pi}{2}-\arg(z-\overline{w})\leq\frac{\pi}{2}.$$
We claim that the argument of $r$ transforms as
$$
\arg(r(\gamma z, \gamma w))=\arg(r(z,w))+\arg(cz+d)-\arg(cw+d)\:.
$$
To see this, notice that both $cz+d$ and $cw+d$ belong either to upper or lower complex half-plane. In particular,  either $\arg(cz+d),\arg(cw+d)\in (0,\pi)$ or $\arg(cz+d),\arg(cw+d)\in (-\pi ,0)$, and it follows that
 $$
 \arg\left(\frac{cz+d}{cw+d}\right) =\arg(cz+d)-\arg(cw+d).
 $$
Consequently, since in \eqref{r-transformation} both values of $r$ have arguments in $(-\frac{\pi}{2},\frac{\pi}{2}]$ and since
 $$
\arg(r(\gamma z, \gamma w))=\arg(r(z,w))+\arg\left(\frac{cz+d}{cw+d}\right) +2\pi l\:
$$
for some $l\in\mathbb Z$, we must have $l=0$.
Since
$$
H_k(z,w)\:=\:H_k(r(z,w))\:=\: \left(\frac{r}{\lvert r\rvert}\right)^{2k_1}\cdot \left(\frac{r}{\lvert r\rvert}\right)^{2k_2},
$$
the claim of the lemma follows trivially for $k=k_1\in\ZZ$, using the definition of $J_{\gamma,k_1}$. For $k=k_2\in(-\frac{1}{2},\frac{1}{2}]$, it follows from  the above by noticing that, for exponent $2k_2\in(-1,1]$, the power is still given by multiplying the argument, $(e^{i\arg(z)})^{2k_2}=e^{2ik_2\arg(z)}$. The lemma follows for arbitrary real $k$ from our choice of the $k$-th power.
\end{proof}
For every $\gamma_1=\left(\begin{smallmatrix}
a_1 & a_2\\
a_3 & a_4
\end{smallmatrix}\right)$ and $\gamma_2=\left(\begin{smallmatrix}
b_1 & b_2\\
b_3 & b_4
\end{smallmatrix}\right)\in\SL_2(\RR)$, write $\gamma_1\gamma_2=\left(\begin{smallmatrix}
c_1 & c_2\\
c_3 & c_4
\end{smallmatrix}\right)$. For every $z\in\HH$, we have
$$
a_3\gamma_2z+a_4=\frac{c_3z+c_4}{b_3z+b_4}
$$
and therefore there exists an integer $w(\gamma_1,\gamma_2)\in\{-1,0,1\}$, which is independent of $z$, such that
\begin{equation}\label{w property}
2\pi w(\gamma_1,\gamma_2)=\arg(a_3\gamma_2z+a_4)+\arg(b_3z+b_4)-\arg(c_3z+c_4).
\end{equation}
The function $\omega_k(\gamma_1,\gamma_2):= \exp(4\pi ik w(\gamma_1,\gamma_2))$ is called a factor system of weight $k$.

Let $(V,\langle\cdot,\cdot\rangle)$ be a $d$-dimensional unitary $\CC$-vector space ($d < \infty$), where the inner product $\langle\cdot,\cdot\rangle$ is semi-linear in the first argument. Let $U(V)$ denote the  unitary group, i.e. the automorphisms $u$ of $V$ respecting the scalar product, $\langle u(v), u(w)\rangle=\langle  v,w\rangle$ for all $v,w\in V$.

\begin{definition}
A \textit{(unitary) multiplier system of weight $k$} on $\widetilde\G$ is a function $\chi: \widetilde\G \to U(V)$ which satisfies the properties:
\begin{enumerate}[label=$(\alph*)$]
\item $\chi(-I)=e^{-2\pi i k}\mathrm{id}_V$ and
\item $\chi(\gamma_1\gamma_2)=\omega_k(\gamma_1,\gamma_2)\chi(\gamma_1)\chi(\gamma_2)$.
\end{enumerate}
\end{definition}

If $\widetilde\G$ contains parabolic elements, then there exists a unitary multiplier system on $\widetilde\G$ for every weight $k\in\RR$; when $\widetilde\G$ does not contain parabolic elements, a unitary multiplier system on $\widetilde\G$ exists for certain rational values of weight $k$, depending on the signature of the group $\Gamma$, see \cite[Proposition 1.3.6]{Fi87}.  From now on, we fix $k\in\RR$ such that there exists a unitary multiplier system $\chi:\widetilde\G\to U(V) $ of weight $k$ on $\widetilde\G$, which we also fix.

\begin{lemma}\label{L:point-pair}
 For every weight $k$ point-pair invariant $\mu$ such that the series
 \begin{equation*}
  S_{\Gamma,\mu}(z,w):=\sum_{\gamma\in\widetilde\Gamma}\chi(\gamma)J_{\gamma,k}(w)\mu(z,\gamma w)
 \end{equation*}
 is absolutely convergent for all $z,w\in\HH$, we have the identity
\begin{equation*}
 S_{\Gamma,\mu} (\eta z,w)J_{\eta,k}(z)^{-1}=\chi(\eta)S_{\Gamma,\mu}(z,w)\
 \end{equation*}
for  all $\eta\in\widetilde\Gamma$.
\end{lemma}
\begin{proof}
We have to prove that
\begin{equation*}
\chi(\eta)\sum_{\gamma\in\widetilde\Gamma}\chi(\gamma)J_{\gamma,k}(w)\mu(z,\gamma w)=
 \sum_{\gamma\in\widetilde\Gamma}\chi(\gamma)J_{\gamma,k}(w)J_{\eta,k}(z)^{-1}\mu(\eta z,\gamma w)
\end{equation*}
for every $\eta$ in $\widetilde{\Gamma}$. Setting $\gamma'=\eta^{-1}\gamma$ and summing over $\gamma'$ instead of $\gamma$ by absolute convergence of the series, the above follows from
the definitions of multiplier system and weight $k$ point-pair invariant, combined with the implication of \eqref{w property} that, for any $w\in\HH$ and any $\eta,\gamma\in\widetilde\Gamma$,
\begin{equation*}
\omega_k(\eta,\gamma)=J_{\eta,k}(\gamma w)J_{\gamma,k}(w)J_{\eta\gamma,k}(w)^{-1}.\qedhere
\end{equation*}
\end{proof}

\subsection{Spectral expansion}\label{sec:spectral exp}

For every $\gamma\in\widetilde\G$, define the linear operator $|[\gamma,k]$ on the space of functions $f:\HH\to V$ by
$$
f|[\gamma,k] (z) := f(\gamma z)J_{\gamma,k}(z)^{-1}.
$$
It is important to notice that $\Delta_k$ commutes with $|[\gamma,k]$, in other words
$$
\Delta_k(f|[\gamma,k])=(\Delta_kf)|[\gamma,k]
$$
for every twice continuously differentiable function $f:\HH \to V$. It follows that, if $f$ is such a function and it additionally satisfies
\begin{equation}\label{chi property}
f|[\gamma,k]=\chi(\gamma)f
\end{equation}
for every $\gamma\in\widetilde\G$, then
\begin{equation}\label{delta property}
(\Delta_kf)|[\gamma,k]=\chi(\gamma)\Delta_kf.
\end{equation}
Notice that if $f_1,f_2:\HH\to V$ are functions satisfying \eqref{chi property} then $\langle f_1,f_2\rangle$ is a $\widetilde\G$-invariant, vector-valued function on $\HH$.
Let $\mathcal{F}$ denote an arbitrary fundamental domain of $\G$. Let $\mathcal{H}_k$ denote the space of (equivalence classes of $\mu_{\mathrm{hyp}}$-almost everywhere equal) $\mu_{\mathrm{hyp}}$-measurable functions $f:\HH\to V$ which satisfy the properties:
\begin{enumerate}[label=$(\alph*)$]
\item $f|[\gamma,k](z)=\chi(\gamma)f(z)$ for all $\gamma\in\widetilde\G$ and
\item $\| f\|^2:= \int_{\mathcal{F}} \langle f,f\rangle d\mu_{\mathrm{hyp}} < \infty$.
\end{enumerate}
It follows that $\mathcal{H}_k$ is a Hilbert space when equipped with the scalar product
$$
(f,g):=\int_{\mathcal{F}}\langle f,g\rangle d\mu_{\mathrm{hyp}}\:.
$$
For all $f_1,f_2\in\mathcal H_k$, the function $\langle f_1,f_2\rangle$ given by the scalar product on $V$ determines an almost everywhere well-defined function on $\mathbb{H}$. 
In particular, when  $V=\CC$, $\langle f,g\rangle=\bar f \cdot g$ and $(f,g)=\int_{\mathcal{F}}\overline{f(z)}g(z) d\mu_{\mathrm{hyp}}(z)$ is the usual $L^2$-scalar product. From now on, the equivalence class of a function $f:\HH\to V$ under the equivalence relation $\mu_{\mathrm{hyp}}$-almost everywhere equal will be denoted by $f$ by abuse of notation.
Moreover, identify $V=\CC^d$, which implies that
\begin{equation}\label{C_inner_product}
\langle x,y\rangle=\sum_{j=1}^d\overline{x}_jy_j
\end{equation}
for every $x=(x_1,\dots,x_d)^t$ and $y=(y_1,\dots,y_d)^t$ in $V$. Here, $X^t$ denotes the transpose of a matrix $X$. With these conventions, measurability, differentiability, integrability, etc. of any function $f:\HH\to V$ are defined component-wise.

The norm on $V$ corresponding to the scalar product $\langle \cdot,\cdot \rangle$ will be denoted by $|\cdot|_V$.
We will at times apply the Hermitian inner product to $d\times d$ matrices, more precisely to $xy^t$ for arbitrary $x,y\in V$. Denote the resulting norm by $|\cdot |_{d\times d}$ and note that
$$|xy^t|_{d\times d}=|x|_V|y|_V.$$

Let $D_k$ denote the set of all twice continuously differentiable functions $f\in\mathcal{H}_k$ such that $\Delta_kf \in \mathcal{H}_k$. The operator $\Delta_k:D_k\to\mathcal{H}_k$ is essentially self-adjoint \cite[Theorem 1.4.5]{Fi87}. Let $\tilde{\Delta}_k: \tilde{D}_k\to \mathcal{H}_k$ denote the unique maximal self-adjoint extension of $\Delta_k$ with $\tilde{D}_k$ as its domain.

In case when $\widetilde\G$ contains parabolic elements, let $\zeta_1,\dots,\zeta_{c_{\G}}$ denote a complete system of representatives of the $\widetilde\G$-equivalence classes of cusps of $\widetilde\G$. Choose matrices $A_1,\dots,A_{c_{\G}}\in\SL_2(\RR)$, such that the stabilizers $\widetilde\G_{\zeta_j}:=\{\gamma\in\widetilde\G\,|\,\gamma\zeta_j=\zeta_j\}$ are generated by $-I$ and $T_j:=A_j^{-1}\left(\begin{smallmatrix}
1&1\\
0&1
\end{smallmatrix}\right)A_j$. Let $m_j$ denote the multiplicity of the eigenvalue $1$ of $\chi(T_j)$. For every $j\in\{1,\dots,c_{\G}\}$, choose an orthonormal basis $\{v_{j1},\dots,v_{jd}\}$ of $V$ such that
$$\chi(T_j)v_{jl}=e^{2\pi i\beta_{jl}}v_{jl},\text{ with }
\begin{cases}
\beta_{jl}=0, & \text{if }1\leq l\leq m_j \text{ and}\\
\beta_{jl}\in(0,1), & \text{if }m_j<l\leq d.
\end{cases}
$$

For every $z\in\HH$ and $s\in\CC$ with $\Re(s)>1$, define the \textit{parabolic Eisenstein series of weight $k$ for the cusp $\zeta_j$, the multiplier system $\chi$ and the eigenvector $v_{jl}$} as the series
\begin{equation}\label{def:Eis series}
E_{jl}(z,s):=\frac{1}{2}\sum_{\gamma\in\widetilde\G_{\zeta_j}\backslash \widetilde\G}\omega_k(A_j,\gamma)^{-1}\chi(\gamma)^{-1}v_{jl}J_{A_j\gamma,k}(z)^{-1}(\Im(A_j\gamma z))^s.
\end{equation}

This series converges uniformly absolutely in $(z,s)\in\HH\times\{s\in\CC\,|\,\Re(s)>1+\varepsilon\}$ for every $\varepsilon>0$, hence it defines a $C^\infty$-function from $\HH\times\{s\in\CC\,|\,\Re(s)>1\}$ to $V$, which is holomorphic in $s$.  It was shown in \cite{roelcke3} that, for every $s\in\CC$ such that $\Re(s)>1$, the series $E_{jl}(\cdot,s)$ is an eigenfunction of $\Delta_k$, with eigenvalue $s(1-s)$:
\begin{equation}\label{eq:Eis laplacian}
\Delta_k E_{jl}(\cdot,s)=s(1-s)E_{jl}(\cdot,s).
\end{equation}
Furthermore, for every fixed $z\in\HH$, the series $E_{jl}(z,\cdot)$ can be extended to a meromorphic function on $\CC$, which is denoted in the same way. This function has only simple poles in the half-plane $\{s\in\CC\,|\,\Re(s)>1/2\}$, which all lie in the interval $(1/2,1]$. It has no poles on the line $\{s\in\CC\,|\,\Re(s)=1/2\}$, from which it follows that $E_{jl}$ is continuous on $\HH\times\{s\in\CC\,|\,\Re(s)=1/2\}$. The Eisenstein series $E_{jl}$ satisfies \eqref{eq:Eis laplacian} in this domain. Recall the following theorem from \cite[pp. 37--38]{Fi87}:
\begin{theorem}[Spectral expansion] \label{thm:spectral expansion}
Every function $f\in\tilde{D}_k$ has an expansion of the following form:
\begin{equation*}
\begin{split}
f(z)&=\sum_{n\geq0}(\vp_n,f)\vp_n(z)+\frac{1}{4\pi}\sum_{j=1}^{c_{\G}}\sum_{l=1}^{m_j}\int_{-\infty}^{\infty}(E_{jl}(\cdot,1/2+it),f)E_{jl}(z,1/2+it)dt,
\end{split}
\end{equation*}
where $(\vp_n)_{n\geq0}$ is a countable orthonormal system  of eigenfunctions of $\tilde{\Delta}_k:\tilde{D}_k\to \mathcal{H}_k$.
The series $\sum_{n\geq0}(\vp_n,f)\vp_n$ converges uniformly absolutely on compact subsets of $\HH$. 
{When $\G$ is cocompact, the second sum on the right hand side of the above equation is identically zero.}
\end{theorem}

\begin{remark} \label{rem:cocompact setting}
In the sequel, whenever we apply the spectral expansion theorem to cocompact $\G$, we will assume that the sum over parabolic elements is identically zero and we will not treat that case separately.
\end{remark}

Let $|k|(1-|k|)=\lambda_0\leq\lambda_1\leq\lambda_2\leq\dots$ 
denote the discrete eigenvalues corresponding to the orthonormal system $(\vp_n)_{n\geq0}$
and write
\begin{equation}\label{def-t_n}
\lambda_n=1/4 +t_n^2
\end{equation}
for every $n$ where $t_n=\sqrt{\lambda_n-1/4}$ and $t_n\in(0, iA]$ when $\lambda_n <\frac14$; here, $A$ is defined as
\begin{equation}\label{eq:A}
A:=\max\left\{1/2,|k|-1/2\right\},
\end{equation}
 and note that $|k|(1-|k|)\geq\frac14 -A^2$.
Every $\lambda_n$ occurs with finite multiplicity $\mu_n$ and the series $\sum_{n\geq0}\lambda_n^{-2}$ converges \cite[Theorem 1.6.5]{Fi87}.

\section{The automorphic kernel}\label{sec:aut kernel}
In this section, we recall the construction of automorphic forms for $\Gamma$ with multiplier system $\chi$, using point-pair kernel functions (i.e. kernel functions depending only upon the hyperbolic distance between the points).

\subsection{Selberg Harish-Chandra transform}
Following \cite[pp. 386--387]{He83}, let $\Phi$ be a real-valued function defined on $[0,\infty)$, four times differentiable in this interval and such that $|\Phi^{(\ell)}(t)|\ll (t+4)^{-\delta-\ell}$, for $\ell=0,1,2,3,4$ and for some $\delta > \max\{1,|k|\}$.
To the weight $k$ point-pair invariant
\begin{equation}\label{point-pair inv function}
k(z,w):=H_k(z,w)\cdot\Phi\left(\frac{|z-w|^2}{\Im(z) \Im(w)}\right),
\end{equation}
where $z,w\in\mathbb H$, we associate the automorphic kernel
\begin{equation}\label{K Gamma}
K_{\Gamma}(z,w):=\frac{1}{2}\sum_{\gamma\in\widetilde{\Gamma}}\chi(\gamma) J_{\gamma,k}(w)k(z,\gamma w),
\end{equation}
which takes values in the endomorphism ring $\mathrm{End}(V)$. Note that $\Phi\left(\frac{|z-w|^2}{\Im(z) \Im(w)}\right)$ is a weight zero point-pair invariant. Due to the bounds on the derivatives of $\Phi$ and to Lemma \ref{L:point-pair}, the automorphic kernel $K_{\Gamma}$ belongs to $\tilde D_k$ as a function of $z$.

The Selberg Harish-Chandra transform $h_\Phi$
of a function $\Phi$ satisfying the conditions stated above can be computed using the following three steps:
\begin{enumerate}
\item[(i)] compute \label{s/h-ch-transform}
$$Q(y)\:=\:\int_{-\infty}^\infty \Phi(y+v^2)\left(\frac{\sqrt{y+4}+iv}{\sqrt{y+4}-iv}\right)^k~dv\:$$
for $y\geq 0$;
\item[(ii)] set $g(u)= Q\left(2(\cosh u -1)\right)$;
\item[(iii)] the Selberg Harish-Chandra transform of $\Phi$ is the Fourier transform of $g$, i.e.
$$
h_\Phi(r)=\int\limits_{-\infty}^{\infty}g(u)e^{iru}du.
$$
\end{enumerate}
The Selberg Harish-Chandra transform exists for complex numbers $r$ with suitably bounded imaginary part.
\begin{remark}
A slightly different, yet equivalent version of the Selberg Harish-Chandra transform of the point-pair invariant is given in \cite[Theorem 1.5]{Fay77}. In the cited text, the automorphic kernel constructed from the point-pair invariant  is defined as
$$
\tilde{K}_{\Gamma}(z,w)=\sum_{\gamma\in\Gamma}\overline{\chi(\gamma)}\left(\frac{c\overline{w}+d}{cw+d}\right)^k \left(\frac{z-\gamma\overline{w}}{\gamma w-\overline{z}}\right)^k g(\cosh(d_{\hyp}(z,\gamma w))),
$$
under the assumption that $g(u)$ is a continuous function of $u>1$, with a majorant $g_1(u)\in L^1 \cap L^2(1,\infty)$ satisfying the following condition: for any $\delta>0$ there exists a constant $m(\delta)> 0$ such that, for all $z,w\in \mathbb{H}$ with $d_{\hyp}(z,w)>\delta$,
$$
g_1(\cosh(d_{\hyp}(z,w)))\leq m(\delta)\:\cdot\!\!\!\!\!\!\!\int\limits_{d_{\hyp}(\zeta,w)<\delta}g_1(\cosh(d_{\hyp}(\zeta,w)))d\mu_{\hyp}(\zeta).
$$
The  Selberg Harish-Chandra transform $h$  of the point-pair invariant function $g$ is given by the formula
$$
h(r)=2\pi\int\limits_{1}^{\infty}g(\cosh(y))\left(\frac{2}{\cosh y+1}\right)^r F\left(r-k,r+k;1;\frac{\cosh y -1}{\cosh y +1}\right)d(\cosh(y)),
$$
where $F(a,b;c;z)$ stands for the (Gauss) hypergeometric function.

In fact, equation \eqref{rel_cosh_u} yields that $\tilde{K}_{\Gamma}(z,w)=\overline{K_{\Gamma}(z,w)}$, where $K_{\Gamma}(z,w)$ is defined by \eqref{K Gamma} with the point-pair invariant function $\Phi$ in definition \eqref{point-pair inv function} given by $\Phi(x)=g(1+\tfrac{x}{2})$; in particular $h=h_\Phi$.
\end{remark}

For a function $h:D\to\CC$, where $D$ is a subset of $\CC$, and a constant $a>0$ define the following conditions:
\begin{enumerate}
\item[(S1)] $h(r)$ is an even function.
\item[(S2)] $h(r)$ is holomorphic in the strip $|\Im(r)|< a+\epsilon$ for some $\epsilon>0$.
\item[(S3)] $h(r)\ll (1 +\lvert r\rvert)^{-2-\delta}$ for some fixed $\delta>0$ as $\lvert r\rvert\to\infty$ in the set of definition of condition (S2).
\end{enumerate}
Choosing $a=A$ as in \eqref{eq:A},
the conditions (S1)--(S3) are actually the assumptions posed on the test function $h$ in the trace formula \cite[Theorem~6.3]{He83}. The following proposition holds:

\begin{proposition}[\protect{\cite[Section~9.7]{He83}}]\label{prop-autom-kernel-expansion}
 Let $A$ be defined as in \eqref{eq:A} and $\lambda_j=1/4+t_j^2$ as in \eqref{def-t_n}. Suppose that the Selberg Harish-Chandra transform $h_\Phi$ exists and satisfies  conditions (S1)--(S3) for $a=A$.
Then the automorphic kernel \eqref{K Gamma} admits a spectral expansion of the form
\begin{equation}\label{K Gamma spectral}
\begin{split}
K_{\G}(z,w) =&\sum_{\lambda_j\geq|k|(1-|k|)}h_\Phi(t_j)\vp_j(z)\overline{\vp_j}(w)^t\\
&+\frac{1}{4\pi}\sum_{j=1}^{c_{\G}}\sum_{l=1}^{m_j}\int_{-\infty}^{\infty}h_\Phi(r)E_{jl}(z, 1/2 + ir)\overline{E_{jl}}(w, 1/2 + ir)^tdr,
\end{split}
\end{equation}
which converges absolutely and uniformly on compacta.
\end{proposition}

When $\G$ is cocompact, according to Remark \ref{rem:cocompact setting}, the second sum on the right hand side of \eqref{K Gamma spectral} is identically zero.

The assumptions on the test function $h$, which ensure the convergence of the series and the integral on the right-hand side of \eqref{K Gamma spectral}, can be relaxed. Namely, we will prove in Section \ref{Sec:Wave dist} that, if the function $h$ satisfies the conditions (S1),
\begin{enumerate}
\item[(S2$'$)] $h(r)$ is well-defined and even for $r\in\RR \cup [-ia, ia]$,
\end{enumerate}
and (S3) in the set of definition of condition (S2$'$) (that is, as $\lvert r\rvert\to\infty$), then the series and integrals on the right-hand side of \eqref{K Gamma spectral} are well-defined and converge absolutely and uniformly on compacta. However, the assumptions (S1), (S2$'$) and (S3) do not imply that the right-hand side of \eqref{K Gamma spectral} represents a spectral expansion of some $L^2$-automorphic kernel for $a=A$.

\subsection{Resolvent kernel and  pre-trace formula}

Let $\rho(\tilde{\Delta}_k)$ denote the resolvent set of $\tilde{\Delta}_k$, i.e. the set of all complex numbers $\lambda$ for which the linear operator $(\tilde{\Delta}_k -\lambda\mathrm{id}_{\tilde{D}_k})^{-1}:\mathcal{H}_k \rightarrow  \tilde{D}_k$ is bounded. According to \cite[pp. 25--27]{Fi87}, the resolvent kernel associated to the operator $\tilde{\Delta}_k$ is the integral kernel of the operator $(\tilde\Delta_k -s(1-s))^{-1}$, defined for all $s\in\CC\setminus\{k-n,-k-n\,|\,n=0,1,2,...\}$ with $\Re(s)>1$ and $z,w\in\HH$ such that $z\neq\gamma w$ for all $\gamma\in\Gamma$ as the automorphic kernel
\begin{equation}\label{def:resolvent kernel}
G_s(z,w):=\frac{1}{2}\sum_{\gamma\in\widetilde{\Gamma}}\chi(\gamma)k_s(\sigma(z,\gamma w))J_{\gamma,k}(w)H_k(z,\gamma w),
\end{equation}
with the point-pair invariant function
$$
k_s(\sigma):=\sigma^{-s}\frac{\Gamma(s-k)\Gamma(s+k)}{4\pi \Gamma(2s)}F(s+k,s-k;2s;\tfrac{1}{\sigma}),
$$
where $\sigma:=\sigma(z,w)$ is defined by \eqref{def_sigma} and $F(\alpha,\beta;\gamma;z)$ denotes the classical Gauss hypergeometric function.

The series on the right-hand side of \eqref{def:resolvent kernel} converges normally in the variables $z,w\in\HH$ such that $z\neq\gamma w$ and $s\in\CC\setminus\{k-n,-k-n\,|\,n=0,1,2,...\}$ with $\Re(s)>1$  with respect to the operator norm in the ring of endomorphisms of $V$.

Recall from \cite[Formula (2.1.4) on p.~46]{Fi87} the pre-trace formula that follows from the computation of the trace of the resolvent kernel $G_s(z,w)$:

\begin{lemma} \label{lemma: pre-trace} For all $s,t\in\CC\setminus\{ k-n,-k-n\,|\,n=0,1,2,\ldots\}$ with $\Re(t),\Re(s)>1$ and $z\in\HH$, we have
\begin{align}\label{pre-trace fla}
&\sum_{n\geq 0} \left(\frac{1}{\lambda_n - \lambda} - \frac{1}{\lambda_n-\mu}\right)|\vp_n(z)|_V^2+\frac{1}{4\pi}\sum_{j=1}^{c_{\G}}\sum_{l=1}^{m_j}\int\limits_{-\infty}^{\infty}\left(\frac{1}{\tfrac{1}{4}+r^2 - \lambda} - \frac{1}{\tfrac{1}{4}+r^2-\mu}\right)\\
\nonumber &\times |E_{jl}(z,\tfrac{1}{2}+ir)|_V^2 dr\\
\nonumber &=-\frac{d}{4\pi}\left(\psi(s+k)+\psi(s-k)-\psi(t+k)-\psi(t-k)\right)\\
\nonumber &\quad +\frac{1}{2}\sum_{\gamma\in\widetilde{\Gamma}\setminus\{\pm I\}}\mathrm{Tr}(\chi(\gamma))\left( k_s(\sigma(z,\gamma z))- k_t(\sigma(z,\gamma z))\right)J_{\gamma,k}(z)H_k(z,\gamma z),
\end{align}
where $\lambda:=s(1-s)$, $\mu:=t(1-t)$ and $\psi(x):=\frac{\Gamma'(x)}{\Gamma(x)}$ is the digamma function.

Moreover, by Dini’s theorem, all the sums and integrals in \eqref{pre-trace fla} converge uniformly for every $s,t$ as above and $z\in\HH$.
\end{lemma}

When $\G$ is cocompact, the sum over cusps on the left hand side of \eqref{pre-trace fla} is identically zero.

\section{Sup-norm bounds for the eigenfunctions associated to discrete eigenvalues}\label{sec:sup norm}

In this section, we use \eqref{pre-trace fla} to derive the sup-norm bounds for the norm $\lvert\vp_n(z)\rvert_V$, when $z\in \mathcal F$. Hence, among others, we need an upper bound for the absolute value of the difference
$g_k(s;z,\gamma z):=k_s(\sigma(z,\gamma z))- k_{s+1}(\sigma(z,\gamma z))$
analogous to the bound derived in  \cite[Lemma~6.2]{FJK19}.  The proof of \cite[Lemma~6.2]{FJK19} could be adopted to our setting when the weight $k$ is not a positive integer or a half-integer. However, we give a direct proof of a better bound, valid for all real weights $k$.
\begin{lemma}
Let $k\in \mathbb{R}$ and let $s>|k|$ be a real number. Then
\begin{equation}\label{k-difference}
|g_k(s; z,\gamma z)|\leq \frac{s}{2\pi(s^2-k^2)}\sigma(z,\gamma z)^{-s} .
\end{equation}
\end{lemma}
\begin{proof}
From the definition of hypergeometric series in terms of the Pochammer symbol $(a)_j:=\Gamma(a+j)/\Gamma(a)$ and the identity $(a+1)_j=(a)_{j+1}/a$, we obtain that
$$
k_{s+1}(\sigma(z,\gamma z))= \sigma(z,\gamma z)^{-s}\frac{\Gamma(s-k)\Gamma(s+k)}{4\pi \Gamma(2s+1)}\sum_{j=0}^{\infty}\frac{(s+k)_{j+1}(s-k)_{j+1} (j+1)}{(j+1)!(2s+1)_{j+1}}\sigma(z,\gamma z)^{-(j+1)},
$$
so that
$$
g_k(s; z,\gamma z)=\sigma(z,\gamma z)^{-s}\frac{\Gamma(s-k)\Gamma(s+k)}{4\pi \Gamma(2s)} \sum_{j=0}^{\infty}\frac{(s+k)_{j}(s-k)_{j}}{j!(2s+1)_{j}}\sigma(z,\gamma z)^{-j}.
$$
Since $\sigma(z,\gamma z)\geq 1$, application of \cite[Formula~9.122.1]{GR07} gives
$$
\sum_{j=0}^{\infty}\frac{(s+k)_{j}(s-k)_{j}}{j!(2s+1)_{j}}\sigma(z,\gamma z)^{-j}\leq \sum_{j=0}^{\infty}\frac{(s+k)_{j}(s-k)_{j}}{j!(2s+1)_{j}}=\frac{\Gamma(2s+1)}{\Gamma(s-k+1)\Gamma(s+k+1)},
$$
which leads to
$$
g_k(s; z,\gamma z)\leq \frac{2s}{4\pi (s-k)(s+k)}\sigma(z,\gamma z)^{-s}.
$$
The proof is complete; note that we have omitted the absolute values because all expressions are positive, due to the fact that $s>|k|$ is real.
\end{proof}

\begin{remark}\rm
In the case where $s=k+\epsilon$, for some $\epsilon\in(0,1)$ and some positive integer $k$, the upper bound from \eqref{k-difference} becomes $\frac{k+\epsilon}{2\pi\epsilon(2k+\epsilon)}\sigma(z,\gamma z)^{-s}$. This is obviously less than $\frac{3}{2\pi \epsilon}\sigma(z,\gamma z)^{-s}$ for all positive integers $k$, hence the bound \eqref{k-difference} is better than the one obtained in \cite[Lemma~6.2]{FJK19} using the representation of the resolvent kernel as an integral transform of the heat kernel.
\end{remark}

Next, we derive the sup-norm bound for the eigenfunctions of the weighted Laplacian associated to discrete eigenvalues $\lambda_j$, $j\geq 0$, and for the integral of the Eisenstein series, when $M$ in non-compact. Throughout this section, identify the surface $M$ with the fundamental domain $\mathcal{F}$. Let $Y>1$ be arbitrary and let $\mathcal{F}_j^{Y}$ denote the neighbourhood of the cusp $\zeta_j$, $j\in\{1,\ldots,c_\Gamma\}$, characterized by
$$
A_j^{-1}\mathcal{F}_j^{Y}=\{z\in\HH\,|\,-1/2\leq \Re(z)\leq 1/2, \Im(z)\geq Y\},
$$
where $A_j$ is the scaling matrix associated to the cusp $\zeta_j$, for every $j\in\{1,\ldots,c_\Gamma\}$. Denote by $\mathcal{F}_Y$ the closure of the complement of $\bigcup\limits_{j=1}^{c_\Gamma}\mathcal{F}_j^{Y} $ with respect to $\mathcal{F}$ (note that $\mathcal{F}= \mathcal{F}_Y$ if $\Gamma$ is cocompact).

We introduce the constant
\begin{equation} \label{constant in sup norm bound}
C(k,M,d):=\frac{d(|k|+2)}{8\pi(|k|+1)}+\left( \frac{|k|+2}{|k|+1}\right)^2 \frac{d}{2\mathrm{vol}_{\mathrm{hyp}}(\mathcal{F})} e^{\frac32\mathrm{diam}_{\mathrm{hyp}}(\mathcal{F})},
\end{equation}
where $\mathrm{diam}_{\mathrm{hyp}}(\mathcal{F})$ denotes the hyperbolic diameter of the fundamental domain $\mathcal{F}$. The constant $C(k,M,d)$ clearly depends upon the surface and the multiplier system, but not on the eigenvalue. With this notation, the following proposition holds:

\begin{proposition} \label{prop:sup_norm_bound}
$(a)$	Let $\vp_j(z)$ be the eigenfunction of the Laplacian $\Delta_k$ associated to the discrete eigenvalue $\lambda_j$. Then
\begin{equation} \label{sup norm bound varphi}
\sup_{z\in \mathcal{F}}|\vp_j (z)|_V \leq \mathcal{C}(k,M,d) |\lambda_j|,
\end{equation}
where the constant $\mathcal{C}(k,M,d)$ depends  on the surface and the multiplier system,
but not on the eigenvalue. When $\lambda_j\geq 3+|k|$, one can take
$$\mathcal{C}(k,M,d)=\left(C(k,M,d)(|k|+2)\right)^{\frac{1}{2}}.$$
$(b)$ In case when $\widetilde\G$ contains parabolic elements, for any $j\in\{ 1,\ldots ,c_{\G}\}$ and $l\in\{1,\ldots ,m_j\}$, the following bound for the parabolic Eisenstein series \eqref{def:Eis series} of weight $k$ for the cusp $\zeta_j$, the multiplier system $\chi$ and the eigenvector $v_{jl}$ holds:
\begin{equation} \label{sup norm bound Eisenstein ser}
\sup_{z\in\mathcal{F}}\int\limits_{-\infty}^{\infty}
\frac{1}{(\frac{1}{4}+r^2 +(|k|+2)^2)^2 -(|k|+2)^2}  |E_{jl}(z,\tfrac{1}{2}+ir)|_V^2 dr\leq \frac{2\pi}{|k|+2} C(k,M,d).
\end{equation}
\end{proposition}
\begin{proof}

Take $s=|k|+2$ and $t=|k|+3$ in Lemma \ref{lemma: pre-trace} (note that $s,t\notin\{k-n,-k-n\,|\,n=0,1,2,\ldots\}$). Start with an upper bound for the right-hand side of \eqref{pre-trace fla}. The sum of the values of digamma functions may be evaluated by applying the functional equation $\psi(z+1)=\psi(z)+z^{-1}$:
$$
\frac{d}{4\pi}\left|\psi(s+k)+\psi(s-k)-\psi(t+k)-\psi(t-k)\right| = \frac{d(|k|+2)}{8\pi(|k|+1)}.
$$
To bound the sum, use inequality \eqref{k-difference}. Recall that $|J_{\gamma,k}(z)H_k(z,\gamma z)|=1$ for all $z$ and $\gamma$ and that $\chi$ is unitary, so that
\begin{equation*}
\begin{split}
&\bigg|\frac{1}{2}\sum_{\gamma\in\widetilde{\Gamma}\setminus\{\pm I\}}\mathrm{Tr}(\chi(\gamma))\left( k_s(\sigma(z,\gamma z))- k_t(\sigma(z,\gamma z))\right)J_{\gamma,k}(z)H_k(z,\gamma z)\bigg|\\
&\leq \sum_{\gamma\in\Gamma\setminus\{I\}}\frac{d(|k|+2)}{8\pi(|k|+1)}\sigma(z,\gamma z)^{-(|k|+2)} .
\end{split}
\end{equation*}
Furthermore, applying  \cite[Lemma~3.7]{FJK19} with $\delta=|k|+2$, we deduce that, for any $Y>1$ and any $z\in \mathcal{F}_Y$,
$$
\sum_{\gamma\in\Gamma\setminus\{I\}}\frac{d(|k|+2)}{8\pi(|k|+1)}\sigma(z,\gamma z)^{-(|k|+2)}\leq \left( \frac{|k|+2}{|k|+1}\right)^2 \frac{dB_{Y}}{2},
$$
where $B_Y=\exp\bigl(\frac{3}{2}\mathrm{diam}_{\mathrm{hyp}}(\mathcal{F}_Y)\bigr)\mathrm{vol}_{\mathrm{hyp}}(\mathcal{F}_Y)^{-1}$. Note that, for every $Y\geq 2$, $B_Y$ is bounded by $\exp\bigl(\frac{3}{2}\mathrm{diam}_{\mathrm{hyp}}(\mathcal{F})\bigr)\mathrm{vol}_{\mathrm{hyp}}(\mathcal{F})^{-1}$. Hence, for all $z\in \mathcal{F}_Y$ and $Y\geq 2$, the right-hand side of \eqref{pre-trace fla} is bounded from above by the constant $C(k,M,d)$ defined in \eqref{constant in sup norm bound}.
\vskip .06in
Now, specialize the pre-trace formula \eqref{pre-trace fla} to either one summand or one integral on the left-hand side.
\vskip .06in
$(a)$ Since there are only finitely many eigenvalues that are less than $3+|k|$, it is sufficient to prove \eqref{sup norm bound varphi} for eigenvalues $\lambda_j\geq 3+|k|$. Therefore, assume that $\lambda_j\geq 3+|k|$.

Our choice of $s$ and $t$ in Lemma \ref{lemma: pre-trace}, together with above computations and the assumption on $\lambda_j$, lead to the inequality
$$\sup_{ z\in\mathcal{F}_Y} |\vp_j(z)|_V^2 \leq C(k,M,d)(|k|+2) \lambda_j^2 ,$$
which holds for all $Y\geq 2$. It remains to extend it to $z\in \mathcal{F}$.

Since all eigenfunctions $\vp_j$ are continuous on $\overline{\mathcal{F}}$ and the area of $\mathcal{F}$ is finite (with the area of the boundary equal to zero, since $\Gamma$ is of the first kind), one deduces that $(z\mapsto |\vp_j(z)|_V)\in L^p(\mathcal{F})$ for all $p\geq 1$ and, more importantly, that
$$
\sup_{z\in \mathcal{F}}|\vp_j (z)|_V = \lim_{p\to \infty} \mu(\mathcal{F})^{-1/p}\| |\vp_j (z)|_V\|_p =\lim_{p\to \infty}\left( \mu(\mathcal{F})^{-1}\int\limits_{\mathcal{F}} |\vp_j(z)|_V^pd\mu_{\hyp}(z)\right)^{1/p},
$$
where $\|\cdot\|_p$ denotes the $L^p$-norm (see e.g. \cite[Formula~(22) on p. 100]{Ch84} for the analogous statement related to eigenfunctions of the Laplacian).

Let $\{Y_n\}_{n\geq 1}$ be an increasing sequence of real numbers bigger than $2$, tending to infinity. For every $p>1$, the monotone convergence theorem applied to the sequence $|\vp_j(z)|^p \cdot 1_{\mathcal{F}_{Y_n}}(z)$, where $1_{\mathcal{F}_{Y_n}}(z)$ denotes the characteristic function of the set $\mathcal{F}_{Y_n}$, yields that
$$
\int\limits_{\mathcal{F}} |\vp_j(z)|_V^pd\mu_{\hyp}(z)=\lim_{n\to\infty} \int\limits_{\mathcal{F}_{Y_n}} |\vp_j(z)|_V^pd\mu_{\hyp}(z)\leq C(k,M,d)^{p/2}(|k|+2)^{p/2} |\lambda_j|^p\mu(\mathcal{F}) .
$$
Therefore,
$$\sup_{z\in \mathcal{F}}|\vp_j (z)|_V \leq C(k,M,d)^{1/2}(|k|+2)^{1/2} |\lambda_j| .$$
\vskip .06in
$(b)$ 
{If $M$ has cusps, fix $j\in\{ 1,\ldots ,c_{\G}\}$ and $l\in\{1,\ldots ,m_j\}$}. The above computations imply that, for $s=|k|+2$ and for all  $Y\geq 2$,
$$
\sup_{z\in\mathcal{F}_Y}\int\limits_{-\infty}^{\infty}
\frac{1}{(\frac{1}{4}+r^2 +s^2)^2 -s^2}  |E_{jl}(z,\tfrac{1}{2}+ir)|_V^2 dr\leq \frac{2\pi}{|k|+2} C(k,M,d).
$$
Proceeding analogously as above, define the function
$$G(z):=\int_{-\infty}^{\infty}
\frac{1}{(\frac{1}{4}+r^2 +s^2)^2 -s^2}  |E_{jl}(z,\tfrac{1}{2}+ir)|_V^2 dr,$$
which is continuous and non-negative on $\overline{\mathcal{F}}$. Applying the monotone convergence theorem to the sequence $G(z)^p \cdot 1_{\mathcal{F}_{Y_n}}(z)$, together with the fact that the sup-norm is the limit of $L^p-$norms, and reasoning as in the proof of part (a), we deduce that
$$
\sup_{z\in \mathcal{F}}\int\limits_{-\infty}^{\infty}
\frac{1}{(\frac{1}{4}+r^2 +s^2)^2 -s^2}  |E_{jl}(z,\tfrac{1}{2}+ir)|_V^2 dr\leq \frac{2\pi}{|k|+2} C(k,M,d).
$$
This completes the proof.
\end{proof}	
	
\section{The wave distribution associated to the weighted Laplacian}\label{Sec:Wave dist}

\subsection{The heat and Poisson kernel}\label{Sec:heat kernel}
In this section, we define the Poisson kernel for the weighted Laplacian $\Delta_k$ via the heat kernel.
For any $t>0$ and $\rho\geq 0$, define the heat kernel

\begin{equation}\label{heat kernel on H}
K_{\mathrm{heat}}(t;\rho):= \frac{\sqrt{2}e^{-t/4}}{(4\pi t)^{3/2}}\int_{\rho}^{\infty}\frac{re^{-r^2/4t}}{\sqrt{\cosh(r)-\cosh(\rho)}}
\mathcal T_{2k}\left(\frac{\cosh(r/2)}{\cosh(\rho/2)}\right)dr,
\end{equation}
where
$$
\mathcal T_{2k}(x)= \frac{1}{2}\left[ (x+\sqrt{x^2-1})^{2k} + (x-\sqrt{x^2-1})^{2k}\right],
$$
for any real $k$. Here the $k$-th powers are chosen as in Lemma~\ref{lemma-H_k-definition}. Note that, for $k\in\frac{1}{2}\mathbb Z$, the function $\mathcal T_{2k}(x)$ coincides with the $2k$-th Chebyshev polynomial.
The hyperbolic heat kernel on $\mathbb H$ is defined by
$$K_{\mathbb H}(t;z,w):=K_{\mathrm{heat}}(t;d_{\hyp}(z,w))~~~~~~~~~(z,w\in \mathbb H).$$
For any $t>0$ and $k\in \mathbb R$, the same argument as in \cite[p. 136]{FJK16} shows that the heat kernel $K_{\mathbb H}(t;\rho)$ is strictly monotonic decreasing with respect to $\rho>0$.
In the spirit of \cite[p. 157]{Fay77}, the hyperbolic heat kernel on $M$ associated to $\Delta_k$ is defined as
\begin{equation}\label{heat kernal on M}
K_{\hyp}(t;z,w):=\frac{1}{2}\sum_{\gamma\in \widetilde\Gamma} \overline{\chi(\gamma)}\left(\frac{c\overline{w}+d}{cw+d}\right)^k\left(\frac{z-\gamma\overline{w}}{\gamma w-\overline{z}}\right)^k K_{\mathbb H}(t;z,\gamma w)~~~~~~~~~(z,w\in \mathcal F).
\end{equation}

\begin{lemma}\label{heatklemma}
	For any $k \in \mathbb R$ and $t>0$, $K_{\hyp}(t;z,w)$ converges absolutely and uniformly on any compact subset of $\mathcal F\times \mathcal F$.
\end{lemma}
\begin{proof}
	Let $U$ be a compact subset of $\mathcal F\times \mathcal F$. Since $\chi$ is a unitary multiplier system and
	$\left|\left(\frac{c\overline{w}+d}{cw+d}\right)^k\left(\frac{z-\gamma\overline{w}}{\gamma w-\overline{z}}\right)^k\right|=1$ for any $\gamma \in \widetilde\Gamma$ and $z,w \in \mathcal F$, it follows that
	\begin{align}
	 \left|\overline{\chi(\gamma)}\left(\frac{c\overline{w}+d}{cw+d}\right)^k\left(\frac{z-\gamma\overline{w}}{\gamma w-\overline{z}}\right)^k K_{\mathbb H}(t;z,\gamma w)\right|_{\mathrm{End}(V)}=\sqrt{d}|K_{\mathbb H}(t;z,\gamma w)|,
	\end{align}
	where $|\cdot|_{\mathrm{End}(V)}$ denotes hermitian norm on $\mathrm{End}(V)$.
	Therefore, in order to prove the absolute and uniform convergence of $K_{\hyp}(t;z,w)$ for any $t>0$ and $(z,w)\in U$, we need to prove the convergence of the series
	$$\sum_{\gamma \in \widetilde\Gamma}K_{\mathbb H}(t;z,\gamma w)$$
	in $\mathbb C$. Introduce  the counting function
	$$N(\rho; z,w):=\#\{\gamma \in \widetilde\Gamma~|~d_{\hyp}(z,\gamma w)<\rho\},$$
	which is defined for any $\rho >0$ and $(z,w)\in U$.
Then \cite{pr09} gives a bound (uniformly for all $(z,w)\in U$) for the function $N(\rho; z,w)$, namely
	$$N(\rho; z,w)=O_{\widetilde\Gamma}(e^{\rho}),$$
	where the implied constant depends only on $\widetilde\Gamma$.
	By Stieltjes integral representation, we have
	\begin{equation}\label{si1}
	\sum_{\gamma \in \widetilde\Gamma}K_{\mathbb H}(t;z,\gamma w)=\int_0^{\infty}K_{\mathrm{heat}}(t;\rho)~dN(\rho;z,w).
	\end{equation}
	Using the fact that $K_{\mathrm{heat}}(t;\rho)$ is a non-negative, continuous and monotonic decreasing function of $\rho$, write
	\begin{equation}\label{si2}
	 \int_0^{\infty}K_{\mathrm{heat}}(t;\rho)~dN(\rho;z,w)=O\left(\int_0^{\infty}K_{\mathrm{heat}}(t;\rho)e^{\rho}~d\rho\right).
	\end{equation}
	Following the idea of the proof of \cite[Proposition 3.3]{FJK16}, we obtain that
	\begin{equation}\label{si3}
	K_{\mathrm{heat}}(t;\rho)\leq e^{-\rho^2/(8t)}G_k(t),
	\end{equation}
	where the function $G_k(t)$ is given by
	$$G_k(t):=\frac{e^{-t/4}}{(4\pi t)^{3/2}}\int_{0}^{\infty} \frac{re^{-r^2/(8t)}}{\sinh(r/2)}e^{kr}dr. $$
	Combining \eqref{si1}, \eqref{si2} and \eqref{si3}, we obtain that
	$$
	\sum_{\gamma \in \widetilde\Gamma}K_{\mathbb H}(t;z,\gamma w)=O_{\widetilde\Gamma}(G_k(t)h(t)),$$
	with $h(t):= \int_0^{\infty}e^{-\rho^2/(8t)}e^{\rho}d\rho$ and where the implied constant depends only on $\widetilde\Gamma$. Hence, the proof is complete.
\end{proof}

Notice that using the notations, introduced in Section 2 of the paper, we can rewrite the heat kernel as
$$K_{\hyp}(t;z,w)=\frac{1}{2}\sum_{\gamma\in \widetilde\Gamma} \overline{\chi(\gamma)}J_{\gamma,k}(w)^{-1}H_k(z,\gamma w)^{-1} K_{\mathbb H}
(t;z,\gamma w)~~~~~~~~~(z,w\in \mathcal F).$$
Now using the fact that $H_k$ is a weight $k$ point-pair invariant, $\chi$ is a multiplier system and the relation
$$J_{\eta,k}(\gamma z)J_{\gamma,k}(z)=\omega_{2k}(\eta,\gamma)J_{\eta \gamma, k}(z), ~~~~(\eta, \gamma \in SL_2(\mathbb R), z \in \mathbb H),$$
one can easily prove the following equations:
$$K_{\hyp}(t;\eta z,w)=K_{\hyp}(t;z,w)J_{\eta, k}(z)^{-1}\chi(\eta)^{-1},$$
$$K_{\hyp}(t;z,\eta w)=J_{\eta, k}(w)\chi(\eta)K_{\hyp}(t;z,w),$$
for $t>0, z,w \in \HH$ and $\eta \in \Gamma$.
The hyperbolic heat kernel $K_{\hyp}(t;z,w)$ admits the spectral expansion
\begin{align}\label{seh}
K_{\hyp}(t;z,w)=&\sum_{\lambda_j\geq \lambda_0}e^{-\lambda_j
	t}\vp_j(z)\overline{\vp_j}(w)^t
\notag\\
&+\frac{1}{4\pi}{\sum_{j=1}^{c_{\G}}\sum_{l=1}^{m_j}\int_{-\infty}^{\infty}e^{-(1/4+r^2)t}E_{jl}(z, 1/2 + ir)\overline{E_{jl}}(w, 1/2 + ir)^tdr.}
\end{align}

\noindent
Let $U\subset \mathcal F\times \mathcal F$ be a compact subset and let $t>0$.
Using the sup-norm bound \eqref{sup norm bound varphi} for the eigenfunctions $\vp_j$ ($j\geq 0$) and
applying the norm arising from the Hermitian inner product to a $d\times d$ matrix in \eqref{C_inner_product},
we obtain that
$$\bigg|\sum_{\lambda_0\leq \lambda_j< 1/4}e^{-\lambda_j
	t}\vp_j(z)\overline{\vp_j}(w)^t
+\sum_{\lambda_j\geq 1/4}e^{-\lambda_j
	t}\vp_j(z)\overline{\vp_j}(w)^t\bigg|_{d\times d} \ll \sum_{\lambda_j\geq \lambda_0}\lambda_j^{-2},$$
where $\lambda_0=|k|(1-|k|)$ and the implied constant depends only on $t$ and on the set $U$, apart from $k$ and $d$  (this is possible because the number of eigenvalues in the interval $[\lambda_0,1/4)$ is finite). The series $\sum_{\lambda_j\geq 1/4} \lambda_j^{-2}$ is convergent (\cite[Theorem 1.6.5]{Fi87}), and so is the series on the right-hand side of \eqref{seh}.
Similarly, using H{\"o}lder's inequality and Proposition \ref{prop:wave-dist}(a) below, we deduce the absolute and uniform convergence of each integral on the right hand side of \eqref{seh}. Therefore, for any $t>0$, the series and the integrals on the right hand side of \eqref{seh}
converge absolutely and uniformly on every compact
subset of $\mathcal F\times \mathcal F$.

From the integral representation of $K_{\mathbb H}(t;d_{\hyp}(z,w))$ and the spectral expansion \eqref{seh}, we deduce that $K_{\hyp}(t;z,w)$ satisfies the following estimates, stated component-wise:
\begin{equation}\label{tz}
K_{\hyp}(t;z,w)=O_{\mathcal F,k}(t^{-3/2}e^{-d^2_{\hyp}(z,w)/4t})~~~~~~ {\rm ~as~} t\rightarrow 0,
\end{equation}
\begin{equation}\label{ti}
K_{\hyp}(t;z,w)=O_{\mathcal F,k}(e^{-\lambda_0 t})~~~~~~~~~~~~~~~~{\rm ~as~} t\rightarrow \infty.
\end{equation}

For every $Z\in \mathbb C$ with Re$(Z)\geq -|k|(1-|k|)$, $z,w \in \mathcal F$ and $u\in \mathbb C$ with Re$(u)\geq 0$, the translated by $-Z$ Poisson kernel $\mathcal P_{M,-Z}(u;z,w)$ is defined as
\begin{equation}\label{pk}
\mathcal P_{M,-Z}(u;z,w):=\frac{u}{\sqrt{4\pi}}\int_{0}^{\infty}K_{\hyp}(t;z,w)e^{-Zt}e^{-u^2/4t}t^{-3/2}dt,
\end{equation}
where the integral is taken component-wise.
This kernel is a fundamental solution of the associated differential operator $\Delta_k+Z-\partial_u^2$. Furthermore, using the spectral expansion
of the heat kernel $K_{\hyp}(t;z,w)$ and the identity (see \cite{JLa03})
$$e^{-a\lambda}=\frac{a}{\sqrt{4\pi}}\int_{0}^{\infty}e^{-t\lambda}e^{-a^2/4t}\frac{dt}{t^{3/2}},~~~ \lambda\geq 0,~~ a\in \mathbb C, ~{\rm Re}(a)\geq 0,$$
we have the following spectral expansion
\begin{align*}
\mathcal P_{M,-Z}(u;z,w)=&\sum_{\lambda_0 \leq \lambda_j <1/4}e^{-u\sqrt{\lambda_j+Z}}\vp_j(z)\overline{\vp_j}(w)^t+
\sum_{\lambda_j\geq 1/4}e^{-u\sqrt{\lambda_j+Z}}\vp_j(z)\overline{\vp_j}(w)^t\notag\\
&+\frac{1}{4\pi}{\sum_{j=1}^{c_{\G}}\sum_{l=1}^{m_j}\int_{-\infty}^{\infty}e^{-u|r|}E_{jl}(z, 1/2 + ir)\overline{E_{jl}}(w, 1/2 + ir)^tdr.}
\end{align*}

Following the steps of the proof of \cite[Theorem 5.2]{JLa03}, using the estimates \eqref{tz}--\eqref{ti} for each component of the heat kernel $K_{\hyp}(t;z,w)$ and the fact that
$$K_{\hyp}(t;z,w)-\sum_{\lambda_j\leq 1/4}e^{-\lambda_j t}=O(e^{-\lambda t})~~~~~~~ {\rm as}~t\rightarrow \infty,$$
where $\lambda$ is the first eigenvalue of $\Delta_k$ bigger than $1/4$, one can deduce that,
for Re$(u)>0$ and Re$(u^2)>0$, the Poisson kernel $\mathcal P_{M,-Z}(u;z,w)$
has an analytic continuation for each entry of the matrix to $Z=-1/4$. The continuation is given by
\begin{align}\label{sepk}
\mathcal P_{M,1/4}(u;z,w)=&\sum_{\lambda_0 \leq \lambda_j <1/4}e^{-u\sqrt{\lambda_j-1/4}}\vp_j(z)\overline{\vp_j}(w)^t+
\sum_{\lambda_j\geq 1/4}e^{-ut_j}\vp_j(z)\overline{\vp_j}(w)^t\notag\\
&+\frac{1}{4\pi}{\sum_{j=1}^{c_{\G}}\sum_{l=1}^{m_j}\int_{-\infty}^{\infty}e^{-u|r|}E_{jl}(z, 1/2 + ir)\overline{E_{jl}}(w, 1/2 + ir)^tdr,}
\end{align}
where $t_j=\sqrt{\lambda_j-1/4}\geq 0$, for $\lambda_j\geq 1/4$ and for $\lambda_j<1/4$ we take the principal branch of $\sqrt{\lambda_j-1/4}$.

When $\G$ is cocompact, the sum over cusps (i.e. the last sum on the right hand side of \eqref{sepk}) is identically zero.

\subsection{The wave distribution and its integral representation}
Let $L^1(\RR)$ denote the space of absolutely integrable functions on $\RR$ and let $C_0^{\infty}(\RR)$ denote its subspace of all infinitely differentiable functions with compact support.
\begin{definition}\label{def:H}
 For any $a\geq 0$, denote by $L^1(\RR,a)$ (resp. $S'(\RR,a)$) the space of \textit{even} functions $g$ in $L^1(\RR)$ (resp. in the Schwartz space on $\RR$) such that  $g(u)\exp(|u|a)$ is absolutely dominated by an integrable function on $\RR$.
 Denote the Fourier transform of every $g\in L^1(\RR,a)$ by
 \begin{equation}\label{H-Fourier}
 H(r,g)\:=\:\int_{-\infty}^\infty g(u)\exp(iru)~du,
 \end{equation}
with the domain extended to all $r\in\CC$ for which it is well-defined.
\end{definition}
Notice that, since $g$ is assumed to be even,
\begin{equation}\label{H-even}
H(r,g)\:=\:2\int_{0}^{\infty}\cos(ur)g(u)~du.
\end{equation}
The following result is a generalization of Lemma 3 in \cite{JvPS16}.
\begin{lemma} \label{lem:cond-g} 
Let $n\geq 3$ be an integer.
\begin{enumerate}[label=$(\alph*)$]
\item Let $g\in L^1(\RR,a)$ be such that $g^{(l)}\in L^1(\RR)$ for $1\leq l\leq n$,
and $\lim_{u\to\infty} g^{(l)}(u)=0$ for $0\leq l\leq n-1$.
Then the Fourier transform $H(r, g)$ is well-defined for $r\in\{ z\in\CC\,|\,|\Im(z)|\leq a\}$ and satisfies the conditions (S1), (S2\,$'$), and (S3) with $\delta=n-2$.
\item Let $\eta>0$ and let $g\in S'(\RR,a+\eta)$ be such that $g^{(j)}(u)\exp(|u|(a +\eta))$ is absolutely bounded by some integrable function on $\RR$ for $1\leq j\leq n-1$. Then the function $H(r, g)$ satisfies the conditions (S1), (S2) for any $0 <\epsilon <\eta$, and (S3)  with $\delta=n-2$.
\item If $g\in S'(\RR,a)$, then $H(r, g)$ is a Schwartz function in $r\in\RR$.
\end{enumerate}
\end{lemma}

\begin{proof}
$(a)$ For all $r\in\{ z\in\CC\,|\,|\Im(z)|\leq a\}$, $u\in\RR$ and $g\in L^1(\RR,a)$,
$\lvert g(u)\exp(iru)\rvert\leq |g(u)|e^{|u|a}$
	is  dominated by an integrable function on $\RR$, and thus $H(r, g)$ is well-defined. It is also even with respect to $r$. Furthermore, for every $r\in\RR$, using the assumptions on the decay of $g^{(l)}$ for $0\leq l\leq n-1$
	and the fact that $g^{(2j+1)}(0) =0$ (since $g$ is even),
we obtain that
	\begin{align*}
		\frac12 H(r,g) &=\left[\frac1r \sin(ur)g(u)\right]_0^{\infty} -\int_0^{\infty} \frac1r \sin(ur)g'(u)du \\
		&=\left[\frac{1}{r^2} \cos(ur)g'(u)\right]_0^{\infty} -\int_0^{\infty} \frac{1}{r^2}\cos(ur)g''(u)du\\
		&=\:\dots\\
&=\left[\frac{(-1)^{1+n/2}}{r^{n-1}} \sin(ur)g^{(n-2)}(u)\right]_0^{\infty} +(-1)^{n/2}\int_0^{\infty} \frac{1}{r^{n-1}}\sin(ur)g^{(n-1)}(u)du\\
&=\left[\frac{(-1)^{1+n/2}}{r^n} \cos(ur)g^{(n-1)}(u)\right]_0^{\infty} +(-1)^{n/2}\int_0^{\infty} \frac{1}{r^n}\cos(ur)g^{(n)}(u)du,\\
	\end{align*}
when $n$ is even. For odd $n$, we obtain a similar series of equations, terminating at
$$\mp\left[\frac{1}{r^{n}} \sin(ur)g^{(n-1)}(u)\right]_0^{\infty} \pm\int_0^{\infty} \frac{1}{r^{n}}\sin(ur)g^{(n)}(u)du.$$
Hence, using the definition of $H(r,g)$ and the integrability conditions, it follows that
$$\frac{(1+|r|)^n}{2}|H(r,g)|\leq c\cdot\sum_{l=0}^n\int_0^{\infty}|g^{(l)}(u)|du\ll 1,$$
for some constant $c$. This proves that $H(r,g)$ satisfies the condition (S3) for $r\in\RR$ with $\delta=n-2$.

$(b)$ By assumption, there is an integrable function $G(u)$ dominating $g(u)\exp(|u|(a+\eta))$ absolutely. In turn,
	\begin{equation*}
	 \lvert g(u)\cos(ur)\rvert\leq G(u)\exp(-(\eta-\epsilon)|u|)
	\end{equation*}
is uniformily bounded in the strip $\lvert \Im(r)\rvert\leq a+\epsilon$ for $0<\epsilon<\eta$. Hence, the integral defining $H(r,g)$ converges absolutely and uniformly on any compact set contained in such a strip, and thus defines a holomorphic function on the open strip $\{r\in\CC\,|\,|\Im(r)| < a+\eta\}$. In particular, conditions (S1) and (S2) are satisfied.
Similarly, for $j=1,\dots,n-1$, the functions $g^{(j)}(u)\cos(ur)$, as well as $g^{(j)}(u)\sin(ur)$, are bounded absolutely and uniformly in $r$ by some integrable functions $G_j(u)\exp(-(\eta-\epsilon)|u|)$. Recalling the computation involving partial integration from part~$(a)$, we obtain (S3) as well.

$(c)$ If $g\in S'(\RR,a)$, then its Fourier transform $H(r,g)$  is a Schwartz function in the variable $r\in\RR$.
\end{proof}

We now define the wave distribution.
\begin{definition}[Wave distribution]\label{def:wave_dist}
Let $z, w\in \mathcal F$. For every $g\in C_0^{\infty}(\RR)$, the \textit{wave distribution} $\mathcal{W}_{M,k,\chi}(z,w)$ applied to $g$ is defined as
\begin{equation} \label{def: wave distr}
\begin{split}
\mathcal{W}_{M,k,\chi}(z,w)(g) :=& \sum_{\lambda_j\geq|k|(1-|k|)}H(t_j,g)\vp_j(z)\overline{\vp_j}(w)^t\\
&+\frac{1}{4\pi}\sum_{j=1}^{c_{\G}}\sum_{l=1}^{m_j}\int_{-\infty}^{\infty}H(r, g)E_{jl}(z, 1/2 + ir)\overline{E_{jl}}(w, 1/2 + ir)^tdr,
\end{split}
\end{equation}
where $\sqrt{\lambda_j-1/4}=t_j\geq 0$ for $\lambda_j\geq 1/4$ and $t_j\in(0, iA]$ when $\lambda_j<\frac14$, where $A$ is defined in \eqref{eq:A}.
\end{definition}

\begin{proposition}\label{prop:wave-dist}
Let $z, w\in \mathcal F$.
\begin{enumerate}[label=$(\alph*)$]
 \item
  For every $g$ as in Lemma \ref{lem:cond-g}$.(a)$ with $a=A$ and $n=4$, the wave distribution $\mathcal{W}_{M,k,\chi}(z,w)$ is well-defined.
 \item
 Let $g\in S'(\RR,A)$ satisfy the conditions of Lemma \ref{lem:cond-g}$.(b)$ with $n=4$. Then $\mathcal W_{M,k,\chi}(z,w)(g)$ represents the automorphic kernel $K_{\Gamma}(z,w)=K_{\Gamma,\Phi}(z,w)$ for the inverse Selberg Harish-Chandra transform $\Phi$ of $H(\cdot,g)$.
\end{enumerate}
\end{proposition}
\begin{proof}
$(a)$ By Lemma \ref{lem:cond-g}, the function $H(r, g)$ is well-defined for all $r\in\CC$ with $|\Im(r)|\leq A$, which implies that the finite sum
$$\sum_{|k|(1-|k|)\leq \lambda_j <\frac14} H(t_j,g)\vp_j(z)\overline{\vp_j}(w)^t$$
converges (recall that $\lim\limits_{j\to\infty}\lambda_j=\infty$). Furthermore, if $\lambda_j\geq\frac14$, then $t_j\in\RR$ and thus $H(t_j,g)\ll (1+|t_j|)^{-4}$ as $j\to\infty$.

Fix $z, w\in \mathcal F$ and observe that $H(t_j,g)\vp_j(z)\overline{\vp_j}(w)^t\in\CC^{d\times d}$. 
By applying the norm $|\cdot |_{d\times d}$ and H\"older's inequality, we obtain that
\begin{multline*}
\sum_{\lambda_j \geq\frac14}|H(t_j,g)\vp_j(z)\overline{\vp_j}(w)^t|_{d\times d}=\sum_{\lambda_j \geq\frac14}|H(t_j,g)||\vp_j(z)|_V |\vp_j(w)|_V\\
\ll\bigg(\sum_{\lambda_j \geq\frac14}|H(t_j,g)||\vp_j(z)|_V^2\bigg)^{1/2}\bigg(\sum_{\lambda_j \geq\frac14}|H(t_j,g)||\vp_j(w)|_V^2\bigg)^{1/2}.
\end{multline*}
Note that, due to the estimate on $|H(t_j,g)|$, each of the factors on the right-hand side can be compared with the sum occurring in the pre-trace formula \eqref{pre-trace fla}. Use \eqref{pre-trace fla} with $s=|k|+2$ and $t=|k|+3$, as in the proof of Proposition \ref{prop:sup_norm_bound}, to obtain the following bound:
\begin{align*}
\sum_{\lambda_j \geq\frac14}|H(t_j,g)||\vp_j(z)|_V^2&
\ll \sum_{\lambda_j \geq\frac14}\frac{1}{(1+|t_j|)^4}|\vp_j(z)|_V^2\\
&\ll \sum_{\lambda_j \geq\frac14}\frac{1}{t_j^4+t_j^2(\frac12 +2s^2)+\frac{1}{16} +\frac12 s^2+s^2(s^2-1)}|\vp_j(z)|_V^2\\
&\leq \frac{C(k,M,d)}{2(|k|+2)}.
\end{align*}

Convergence of the integral (uniform on compact subsets of $\mathcal{F}\times \mathcal{F}$) can be proved in a similar way, completing the proof.

$(b)$ Because the properties of $g$ imposed by assumption imply those claimed in part $(a)$, the wave distribution is well-defined by the spectral expansion with coefficients $H(r,g)$. The claim follows from Proposition~\ref{prop-autom-kernel-expansion}, once we establish that $H(r,g)$ belongs to the image of the of the Selberg Harish-Chandra transform.
Hence, we have to invert steps (i)--(iii) of page~\pageref{s/h-ch-transform}.
The inverse of $H(r,g)$ under Fourier transformation (iii) is trivially $g$. Since $g$ is an even $C^\infty$-function, the inverse $Q:\RR^+\to\CC$ of $g$ under (ii)  exists and it is given by
\begin{equation*}
 Q(y)\:=\:g\left( 2\log \left(\frac{1}{2}\sqrt{y+4}+\frac{1}{2}\sqrt{y}\right)\right).
\end{equation*}
It belongs to $C^\infty(\RR^+)$. Since $g(u)\exp((A+\eta)u)\to 0$ for $u\to\infty$, we obtain that
\begin{equation*}
 Q(y)\ll\exp\left(-(A+\eta)2\log \left(\frac{1}{2}\sqrt{y+4}+\frac{1}{2}\sqrt{y}\right)\right),
\end{equation*}
i.e.
\begin{equation*}
 Q(y)\ll(y+4)^{-(A+\eta)}.
\end{equation*}
Similarly, for its first derivative
\begin{equation*}
 Q'(y)\:=\: g'\left(2\log \left(\frac{1}{2}\sqrt{y+4}+\frac{1}{2}\sqrt{y}\right)\right)\cdot\frac{1}{\sqrt{y(y+4)}}
\end{equation*}
we find that
\begin{equation*}
 Q'(y)\ll(y+4)^{-(A+\eta)-1},
\end{equation*}
and for its second one
\begin{equation*}
\begin{split}
  Q''(y)\:=\:&g''\left(2\log \left(\frac{1}{2}\sqrt{y+4}+\frac{1}{2}\sqrt{y}\right)\right)\cdot\frac{1}{\sqrt{y(y+4)}}\\
  &-g'\left(2\log \left(\frac{1}{2}\sqrt{y+4}+\frac{1}{2}\sqrt{y}\right)\right)\cdot\frac{y+2}{(y(y+4))^{\frac{3}{2}}}
  \end{split}
\end{equation*}
we obtain the estimate
\begin{equation*}
 Q''(y)\ll(y+4)^{-(A+\eta)-2}.
\end{equation*}
By \cite[pp.~455--457]{He76}, the inverse of $Q$ under (i) is given by
\begin{equation}\label{Phi-inverse}
 \Phi(x)\:=\:-\frac{1}{\pi}\int\limits_{-\infty}^\infty Q'(x+t^2)\left(\frac{\sqrt{x+4+t^2}-t}{\sqrt{x+4+t^2}+t}\right)^k~dt\:,
\end{equation}
where $\Phi\in C^1(\RR^+)$ satisfies
\begin{equation}\label{cond-Phi-inverse}
 \lvert\Phi(x)\rvert\ll(x+4)^{-\alpha}\quad \textrm{ and }\quad \lvert\Phi'(x)\rvert\ll(x+4)^{-\alpha-1},
\end{equation}
for some $\alpha>\max\{1,|k|\}$. We have to show that the integral in (\ref{Phi-inverse}) is $C^1$ and satisfies the two conditions (\ref{cond-Phi-inverse}).

Let $\beta=A+\eta+1$. The bound on $Q'$ together with its differentiability allows us to conclude that the first condition of (\ref{cond-Phi-inverse}) holds, once we prove that
\begin{equation}\label{pre-bound}
\frac{-1}{\pi}\int\limits_{0}^\infty (x+4+t^2)^{-\beta}\left[\left(\frac{\sqrt{x+4+t^2}-t}{\sqrt{x+4+t^2}+t}\right)^k + \left(\frac{\sqrt{x+4+t^2}-t}{\sqrt{x+4+t^2}+t}\right)^{-k}\right]~dt\ll (x+4)^{-(\beta-1/2)}.
\end{equation}
Let $x_1=x+4$ and introduce the change of variables $y=\frac{\sqrt{x_1+t^2}-t}{\sqrt{x_1+t^2}+t}$ in the integral on the left-hand side of \eqref{pre-bound}. Using
\begin{equation*}
 t=\frac{\sqrt{x_1}(1-y)}{2\sqrt y}\:,\quad x_1+t^2=\frac{x_1(1+y)^2}{4y}\:,\quad dt=-\frac{\sqrt{x_1}(1+y)}{4y^\frac{3}{2}}dy\:,
\end{equation*}
the integral becomes
$$
-\frac{1}{\pi} 4^{\beta-1}x_1^{-(\beta-1/2)}\int\limits_0^1(1+y)^{-2\beta+1}(y^{\beta+k-3/2} + y^{\beta-k-3/2})~dy\:,
$$
which is finite if and only if $\beta-|k|-3/2>-1$. This inequality in turn holds, due to our choice of $A$ and $\eta>0$. This proves \eqref{pre-bound} and the first part of \eqref{cond-Phi-inverse}.

Next, we prove that $\Phi(x)$ is $C^1$ and that the second bound of \eqref{cond-Phi-inverse} holds. In order to prove that $\Phi(x)$ is $C^1$, it is sufficient to show that the integrand in \eqref{Phi-inverse} is differentiable in $x$ and that the derivative of the integrand is bounded by some integrable function. If so, then
$$
\Phi'(x)=\frac{-1}{\pi}\int\limits_{-\infty}^\infty\frac{d}{dx}\left( Q'(x+t^2)\left(\frac{\sqrt{x+4+t^2}-t}{\sqrt{x+4+t^2}+t}\right)^k \right)~dt\:.
$$
Differentiability of the integrand with respect to $x$ is obvious, so it remains to prove that
\begin{equation}\label{pre-bound for deriv}
\frac{-1}{\pi}\int\limits_{-\infty}^\infty\frac{d}{dx}\left( Q'(x+t^2)\left(\frac{\sqrt{x+4+t^2}-t}{\sqrt{x+4+t^2}+t}\right)^k \right)~dt \ll (x+4)^{-(\beta+1/2)}.
\end{equation}
Analogously to the above, starting with the bound for $Q''$, we immediately deduce that
$$
\frac{-1}{\pi}\int\limits_{-\infty}^\infty Q''(x+t^2)\left(\frac{\sqrt{x+4+t^2}-t}{\sqrt{x+4+t^2}+t}\right)^k~dt \ll (x+4)^{-(\beta+1/2)}.
$$
Therefore, to prove \eqref{pre-bound for deriv} and complete the proof of part $(b)$, it suffices to show that
$$
\frac{-1}{\pi}\int\limits_{0}^\infty Q'(x+t^2)\frac{d}{dx}\left[\left(\frac{\sqrt{x+4+t^2}-t}{\sqrt{x+4+t^2}+t}\right)^k + \left(\frac{\sqrt{x+4+t^2}-t}{\sqrt{x+4+t^2}+t}\right)^{-k}\right]~dt \ll (x+4)^{-(\beta+1/2)}.
$$
This can be done analogously to the proof of the first bound in \eqref{cond-Phi-inverse}, i.e. take $x_1=x+4$ and change variables to $y=\frac{\sqrt{x_1+t^2}-t}{\sqrt{x_1+t^2}+t}$. Using the bound for $Q'$, it follows that the above integral is bounded by
$$
\frac{k}{\pi}x_1^{-\beta-1/2}4^{\beta-1}\int\limits_0^1(1+y)^{-2\beta}(1-y)\left( y^{\beta+k-3/2} - y^{\beta-k-3/2}\right)\ll x_1^{-\beta-1/2},
$$
because $\beta-|k|-3/2>-1$.
\end{proof}

\begin{theorem}\label{thm:integral kernel}
	Let $z,w \in \mathcal F$ be such that $z\ne w$. Then there exists a continuous $d\times d$ matrix-valued function $W(u;z,w)$ on $\mathbb R^+$ such that the following hold:
	\begin{enumerate}[label=(\alph*)]
		\item
		 $W(u;z,w)=\sum_{\lambda_j<1/4}e^{u\sqrt{1/4-\lambda_j}}(\sqrt{1/4-\lambda_j})^{-4}\vp_j(z)\overline{\vp_j}(w)^t+O(u^4)$ as $u\rightarrow \infty$;
		\item
		$W^{(j)}(u;z,w)=O(u^{4-j})$ as $u\rightarrow 0 ~(j=0,1,2,3)$;
		\item
		For any $g\in S'(\mathbb R,A)$ such that $g^{(j)}(u)\exp(Au)$
has a limit as $ u\rightarrow \infty$ and is bounded by some integrable
		function on $\mathbb{R}$ for $j=0,1,2,3,4$, we have
		\begin{equation}\label{ir}
		\mathcal{W}_{M,k,\chi}(z,w)(g)=\int_0^{\infty}W(u;z,w)g^{(4)}(u)du.
		\end{equation}
	\end{enumerate}
\end{theorem}

\begin{proof}
	For every $\zeta \in \mathbb{C}\setminus \{0\}$ and $t\neq 0$ in the strip $\{t\in\CC:{|\rm Im}(t)|\leq A\}$, define
$$w(\zeta,t):=\frac{e^{-t\zeta}-\sum_{l=0}^{3}h_{l,3}(\sin t)(-\zeta)^l}{t^4}$$
	and set $w(\zeta,0):=\lim_{t\rightarrow 0}w(\zeta,t)$.
	In the above definition, $h_{l,3}(x)$ is a polynomial of degree at most 3 such that, for all $ t\in \mathbb R$ such that $t\to 0$,
	$$h_{l,3}(\sin t)=\frac{t^l}{l!}+O(t^4).$$
	For the explicit construction of $h_{l,3}(x)$, see the proof of \cite[Theorem 2]{CJS18}.
	It is easy to see that $w(\zeta,0)$ is well-defined and equal to ${\frac{\zeta^4}{24}}$.
	If $t \in \mathbb R^+$, then, using the above estimate of $h_{l,3}$ for $0\leq l\leq 3$, we obtain that
	$$ w(\zeta,t)=O_{\zeta}(1)~~\text{as}~~t~\rightarrow 0.$$
	Furthermore, if $\rm{Re}(\zeta)\geq 0$ and $t \in \mathbb R^+$, then $e^{-t \zeta }$ and $\sin(t)$ are bounded functions as $t\rightarrow \infty$. Therefore,
	\begin{equation}\label{wti}
	w(\zeta,t)=O_{\zeta}(t^{-4})~~\text{as}~~t~\rightarrow +\infty.
	\end{equation}
	
Note that the above estimates also hold for $w^{(j)}(\zeta,t)$, the $j$-th derivative of $w(\zeta,t)$ with respect to $\zeta$, for $j=1,2,3,4$.
	For $z,w \in \mathcal F$ with $z\neq w$ and $\zeta \in \mathbb C$ with Re$(\zeta)\geq 0$, define the following matrix-valued function:
	\begin{align}\label{Wtilde}
	\widetilde{W}(\zeta;z,w)=&\sum_{\lambda_j\geq \lambda_0}w(\zeta,t_j)\vp_j(z)\overline{\vp_j}(w)^t\notag \\
	& +
	 \frac{1}{4\pi}\sum_{j=1}^{c_{\G}}\sum_{l=1}^{m_j}\int_{-\infty}^{\infty}w(\zeta,|r|)E_{j,l}(z,1/2+ir)\overline{E_{j,l}}(w,1/2+ir)^tdr,
	\end{align}
where $t_j=\sqrt{\lambda_j-1/4}$ for every $j\geq 0$ is given by the principal branch of the square root. Recall that, if $\lambda_j<1/4$, then  $t_j\in (0, iA]$, otherwise $t_j\geq 0$. Also, in case when $\G$ is cocompact, the second sum on the right hand side of \eqref{Wtilde} is identically zero.

Following the same reasoning as in the proof of Proposition \ref{prop:wave-dist}, part $(a)$ (i.e., using the H\"older inequality, the estimate \eqref{wti}, comparing with the pre-trace formula \eqref{pre-trace fla} and using the bounds obtained in Proposition \ref{prop:sup_norm_bound}),
it is clear that for $\Re(\zeta)\geq 0$ the series
	$$\sum_{\lambda_j\geq \lambda_0}w(\zeta,t_j)\vp_j(z)\overline{\vp_j}(w)^t$$
	converges absolutely and uniformly on $\mathcal F\times \mathcal F$.
The same holds for 	the integral
	$$\int_{-\infty}^{\infty}w(\zeta,|r|)E_{j,l}(z,1/2+ir)\overline{E_{j,l}}(w,1/2+ir)^tdr,$$
	for any pair $(j,l)$ with $1\leq j\leq c_{\Gamma}$ and $1\leq l\leq m_j$, when $\G$ is non-compact.

	Therefore, for every arbitrary fixed $\zeta \in \mathbb C$ with Re$(\zeta)\geq 0$,
	$\widetilde{W}(\zeta;z,w)$ is a well-defined matrix-valued function which converges absolutely and uniformly on $\mathcal{F}\times \mathcal{F}$.
	Moreover, any of the first $4$ derivatives with respect to $\zeta$ of $\widetilde{W}(\zeta;z,w)$ converges uniformly and absolutely, provided
	$\rm{Re}(\zeta)>0$. Therefore, term by term differentiation is valid. By differentiating component-wise four times and using the fact that
	$\frac{d^4}{d\zeta^4}w(\zeta,t)=e^{-\zeta t}$, we obtain that
	\begin{align}
	\frac{d^4}{d\zeta^4}\widetilde{W}(\zeta;z,w)=&\sum_{\lambda_0\leq \lambda_j<1/4}e^{-\zeta \sqrt{\lambda_j-1/4}}\vp_j(z)\overline{\vp_j}(w)^t
	+\sum_{\lambda_j\geq 1/4}e^{-\zeta t_j}\vp_j(z)\overline{\vp_j}(w)^t \notag \\
	& +
	\frac{1}{4\pi}\sum_{j=1}^{c_{\G}}\sum_{l=1}^{m_j}\int_{-\infty}^{\infty}e^{-\zeta |r|}E_{jl}(z,1/2+ir)\overline{E_{jl}}(w,1/2+ir)^tdr \notag\\
	=&\ \mathcal{P}_{M,1/4}(\zeta;z,w),\notag
	\end{align}
	where $\mathcal{P}_{M,1/4}$ is defined in \eqref{pk}.
	For Re$(\zeta)>0$, define
	\begin{equation}
	\mathcal{P}^{(k)}(\zeta;z,w)=
	\begin{dcases}
	\mathcal{P}_{M,1/4}(\zeta;z,w), &\text{if $ k=0$ and}\\
	\int_{0}^{\zeta}   \mathcal{P}^{(k-1)}(\xi;z,w)d\xi, & \text{if $k\geq 1.$}
	\end{dcases}
	\end{equation}
	In the above definition, the integral is taken component-wise over a ray contained in the upper half-plane Re$(\zeta)>0$.
	With this definition, we have
	\begin{equation}\label{q}
	\mathcal{P}^{(4)}(\zeta;z,w)=\widetilde{W}(\zeta;z,w)+q(\zeta;z,w),
	\end{equation}
	where $q(\zeta;z,w)$ is a $d\times d$ matrix-valued function consisting of degree 3 polynomials in $\zeta$, with coefficients depending on $z$ and $w$
	at each component.
	For $z\neq w$ and $\zeta \rightarrow 0$, the function $\mathcal{P}^{(0)}(\zeta;z,w)$ has a limit; therefore,
	\begin{equation}\label{ep}
	\mathcal{P}^{(k)}(\zeta;z,w)=O(\zeta^k) ~~\text{as}~~\zeta\rightarrow 0.
	\end{equation}
	For every $u\in \mathbb{R}^+$, define
	\begin{equation}\label{mf}
	 W(u;z,w)=\frac{1}{2i}\left[\left(\widetilde{W}(iu;z,w)+q(iu;z,w)\right)+\left(\widetilde{W}(-iu;z,w)+q(-iu;z,w)\right)\right].
	\end{equation}
	We claim that the function $W(u;z,w)$ satisfies all the required conditions given in the statement.
	Using the spectral expansion \eqref{sepk} of the Poisson kernel $\mathcal{P}_{M,1/4}(\zeta;z,w)$ and integrating it four times,
	we obtain the property $(a)$.
	Assertion $(b)$ follows using the bound \eqref{ep} in \eqref{q}. For a given $g$ as in the statement, we can derive assertion $(c)$ using integration by parts four times on the right-hand side of \eqref{ir}.
\end{proof}

\section{The basic automorphic kernel}\label{sec:kernel construction}

In this section, we study two automorphic kernels, namely the basic and the geometric automorphic kernels. After defining the basic automorphic kernel $K_s(z,w)$ for any $z,w \in \mathcal F$ in an appropriate complex half $s$-plane in terms of the wave distribution applied to a test function, we prove that it has a meromorphic continuation to the whole complex $s$-plane. Then we introduce the geometric automorphic kernel $\tilde{K}_s(z,w)$ and show that $K_s(z,w)=\tilde{K}_s(z,w)$ for $\Re(s)>\max\{1,|k|\}$, thus also obtaining the meromorphic continuation of $\tilde{K}_s(z,w)$ to the whole complex $s$-plane.

\subsection{Construction and meromorphic continuation of the basic automorphic kernel}
For $z,w \in \mathcal F$ and $s\in \CC$ with $\Re(s)>\max\{1,|k|\}$, we define the basic automorphic kernel $K_s(z,w)$ by
\begin{equation}\label{basic aut kernel}
K_s(z,w):=\frac{\Gamma(s-\frac{1}{2})}{\Gamma(s)}\mathcal W_{M,k,\chi}(z,w)\left(\cosh(u)^{-(s-\frac{1}{2})}\right).
\end{equation}
Here, $\mathcal W_{M,k,\chi}(z,w)$ is the wave distribution and it is applied to the test function
\begin{equation}\label{g_s-fun}
g_s(u)=\frac{\Gamma(s-\frac{1}{2})}{\Gamma(s)}\cosh (u)^{-(s-\frac{1}{2})}.
\end{equation}
Notice that $g_s(u)$ satisfies the conditions in Lemma \ref{lem:cond-g}.$(a)$ with $a=A$ 
and $n=4$ where $A$ is defined in \eqref{eq:A}. Thus, $K_s(z,w)$ is well-defined by Proposition \ref{prop:wave-dist}.$(a)$.

\begin{lemma}\label{H-relation}
	For all $s\in \CC$ with $\Re(s)>\max\{1,|k|\}$, $n \in \mathbb{N} $, and $r \in\mathbb R\cup [-Ai,Ai]$, the Fourier transform $H(r,\cdot)$ (see \eqref{H-Fourier}) of $g_s$  given by \eqref{g_s-fun} satisfies the functional equation
	\begin{equation}\label{H-rel.id.}
	 H(r,g_s)=\frac{2^{-2n}(s)_{2n}}{\left(\frac{s}{2}-\frac{1}{4}-\frac{ir}{2}\right)_n\left(\frac{s}{2}-\frac{1}{4}+\frac{ir}{2}\right)_n}H(r,g_{s+2n}),
	\end{equation}
	where $(.)_n$ denotes the Pochhammer symbol.
\end{lemma}
\begin{proof}
	Let $n \in \mathbb{N}$ and $s \in \CC$ with $\Re(s)>\max\{1,|k|\}$. Definitions \eqref{H-even} and \eqref{g_s-fun} imply that
	\[
	 H(r,g_{s})=\frac{2\G(s-\frac{1}{2})}{\G(s)}\int_{0}^{\infty}\cos(ur)\cosh(u)^{-(s-1/2)}du.
	\]
	When $r\in\RR\setminus\{0\}$ or $r\in[-Ai,Ai]\setminus\{0\}$, or $r=0$, we have (see \cite[Formulas 3.985.1, 3.512.1 and 3.512.2, respectively]{GR07})
	\begin{equation}\label{H-int-id}
	 \int_{0}^{\infty}\cos(ur)\cosh(u)^{-v}du=\frac{2^{\nu-2}}{\G(\nu)}\G\left(\frac{\nu-ir}{2}\right)\G\left(\frac{\nu+ir}{2}\right),
	\end{equation}
	where $\Re(\nu)>A$.  Hence, for $r\in \RR\cup[-Ai,Ai]$, using \eqref{H-int-id} with $\nu=s-1/2$, $\nu=s-1/2+2n$ and the Pochhammer symbol $(s)_n=\frac{\G(s+n)}{\G(s)}$, we obtain the identity \eqref{H-rel.id.}.
\end{proof}

The functional equation \eqref{H-rel.id.} enables us to deduce the meromorphic continuation of the kernel $K_s(z,w)$ to the whole complex $s$-plane.

\begin{theorem}\label{merom.cont.}
	For any $z,w \in \mathcal F$, the basic automorphic kernel $K_s(z,w)$ admits a meromorphic continuation to the whole complex $s$-plane.  The possible poles of the function $\G(s)\G(s-1/2)^{-1}K_s(z,w)$ are located at the points $s=1/2\pm it_j-2n$, where $n\in\NN$ and $\lambda_j=1/4+t_j^2$ is a discrete eigenvalue of $\Delta_k$. When $M$ is non-compact, possible poles of $K_s(z,w)$ are also located at the points $s=1-\rho-2n$, where $n\in\NN$ and $\rho\in(1/2,1]$ is a pole of the parabolic Eisenstein series $E_{jl}(z,s)$, and the points $s=\rho-2n$, where $n\in\NN$ and $\rho$ is a pole of $E_{jl}(z,s)$ with $\Re(\rho)<1/2$.
\end{theorem}
\begin{proof} The proof we present here follows closely the proof of \cite[Theorem 10]{JvPS16}. We assume $M$ is non-compact; in case of cocompact $\G$, the sums over cusps below are identically zero and there are no poles stemming from poles of the Eisenstein series.

Let $B:=\max\{1,|k|\}$.  First we prove that $K_s(z,w)$ has a meromorphic continuation to the half-plane $\Re(s)>B-2n$ for any $n\in\NN$.  For $s\in \CC$ with $\Re(s)>B$ we use the wave representation \eqref{def: wave distr} of $K_s(z,w)$ to obtain 
	\begin{equation} \label{K_s wave rep.}
	\begin{split}
	K_s(z,w)=& \sum_{\lambda_j\geq|k|(1-|k|)}H(t_j,g_s)\vp_j(z)\overline{\vp_j}(w)^t\\
	&+\frac{1}{4\pi}\sum_{j=1}^{c_{\G}}\sum_{l=1}^{m_j}\int_{-\infty}^{\infty}H(r, g_s)E_{jl}(z, 1/2 + ir)\overline{E_{jl}}(w, 1/2 + ir)^tdr,
	\end{split}
	\end{equation}
	where $\lambda_j=1/4+t_j^2$.
	
	Letting $h_n(r,s):={\left(\frac{s}{2}-\frac{1}{4}-\frac{ir}{2}\right)_n\left(\frac{s}{2}-\frac{1}{4}+\frac{ir}{2}\right)_n}$ and using formula \eqref{H-rel.id.} in \eqref{K_s wave rep.}, we get
	\begin{equation} \label{K_s wave rep.2}
	\begin{split}
	\frac{2^{2n}\G(s)}{\G(s+2n)}K_s(z,w)=& \sum_{\lambda_j\geq|k|(1-|k|)}\frac{H(t_j,g_{s+2n})}{h_n(t_j,s)}\vp_j(z)\overline{\vp_j}(w)^t\\
	 &+\frac{1}{4\pi}\sum_{j=1}^{c_{\G}}\sum_{l=1}^{m_j}\int_{-\infty}^{\infty}\frac{H(r,g_{s+2n})}{h_n(r,s)}E_{jl}(z, 1/2 + ir)\overline{E_{jl}}(w, 1/2 + ir)^tdr.
	\end{split}
	\end{equation}
	
	It can be easily seen that $(a)_n=\prod _{j=0}^{n-1}(a+j)$.  This implies that $h_n(r,s)\sim r^{2n}$ as $r\rightarrow \infty$.  Hence, the series in \eqref{K_s wave rep.2} arising from the discrete spectrum is locally absolutely and uniformly convergent as a function of $s$ for $\Re(s)>B-2n$ away from the poles of $h_n(r,s)^{-1}$, i.e. away from the zeros of $h_n(r,s)$.  Using  $(a)_n=\prod _{j=0}^{n-1}(a+j)$, we calculate the zeros of $h_n(r,s)$ for $\Re(s)>B-2n$, which occur at the points $s=1/2\pm it_j-2m$ for $m=0,\dots,n-1$.
	
	Next, we prove the meromorphic continuation of the integral coming from the continuous spectrum in \eqref{K_s wave rep.2}.  First, substitute $r\mapsto 1/2+ir$ so that the integral is now over the vertical line whose real part is $1/2$ and 	observe that as a function of $s\in\mathbb{C}$ the integral, denote it by $I_{1/2,jl}(s)$, is holomorphic for $s\in \CC$ with $\Re(s)>B-2n$ satisfying $\Re(s)\neq 1/2-2m$, where $m=0,\dots,n-1$.
	In order to get the meromorphic continuation of this function across the lines $\Re(s)=1/2-2m$, we will use the same method applied in the proofs of \cite[Theorem 2]{JKvP10} and \cite[Theorem 10]{JvPS16} or in \cite{vP10}.
	
	As a first step, let $m=0$ and choose $\epsilon>0$ sufficiently small to guarantee that $E_{jl}(z,s)$ has no poles in the strip $1/2-\epsilon<\Re(s)<1/2+\epsilon$.  For $s\in\mathbb{C}$ with $1/2 < \Re(s) < 1/2 + \epsilon$, we apply the residue theorem to the function $I_{1/2,jl}(s)$ to obtain the meromorphic continuation $I_{1/2,jl}^{(1)}(s)$ of it in the strip $1/2-\epsilon<\Re(s)<1/2+\epsilon$.  Then, assuming $1/2-\epsilon<\Re(s)<1/2$ and using the residue theorem again, we get the meromorphic continuation $I_{1/2,jl}^{(2)}(s)$ of the integral $I_{1/2,jl}^{(1)}(s)$ to the strip $-3/2<\Re(s)<1/2$. Finally, adding the formulas coming from the applications of the residue theorem we obtain the meromorphic continuation of the integral $I_{1/2,jl}(s)$ to the strip $-3/2<\Re(s)\leq 1/2$.  Similarly, we can get the meromorphic continuation of the integral $I_{1/2,jl}(s)$ to the strip $-3/2-2m<\Re(s)\leq 1/2-2m$ for $m=1,\dots,n-1$ by repeating this two-step process.  The poles that arise in this process are at $s=1-\rho-2m$, where $\rho$ is a pole of the Eisenstein series $E_{jl}(z,s)$ belonging to the line segment $(1/2,1]$, and at $s=\rho-2m$, where $\rho$ is a pole of the Eisenstein series $E_{jl}(z,s)$ such that $\Re(s)<1/2$, and $m=0,\dots,n-1$.
	
This completes the proof of the meromorphic continuation of $K_s(z,w)$ to the whole $s$-plane, as $n\in\NN$ was chosen arbitrarily.
\end{proof}

\subsection{The geometric automorphic kernel}

For $\Re(s)$ sufficiently large and for any two points $z,w\in \mathcal F$, define the geometric automorphic kernel by
\begin{equation}\label{geometric aut kernel}
\begin{split}
\tilde K_s(z,w):=&\frac{\Gamma(s-k)\Gamma(s+k)}{\sqrt{2\pi}\Gamma(s)^2}\sum_{\gamma\in\widetilde{\Gamma}}\chi(\gamma)\cosh(d_{\hyp}(z,\gamma w))^{-s}\\
&\times F(-k,k;s;(1+\cosh(d_{\hyp}(z,\gamma w)))^{-1})J_{\gamma,k}(w)H_k(z,\gamma w),
\end{split}
\end{equation}
where $F(-k,k;s;(1+\cosh(d_{\hyp}(z,\gamma w)))^{-1})$ stands for the (Gauss) hypergeometric function.

Note that it is possible to extend the above definition to $z,w\in\HH$. Then, for any fixed $w\in \HH$, the function $\tilde K_s(z,w)$ can be viewed as a map from $\HH$ to $\mathrm{End}(V)$ (which can be identified with $\CC^{d\times d}$). The following proposition shows that (for sufficiently large $\mathrm{Re}(s)$) for any fixed $w\in \HH$, the columns of the $d\times d$ matrix $\tilde K_s(z,w)$ belong to the space $\mathcal{H}_k$ (see Section \ref{sec:spectral exp}), when viewed as maps from $\HH$ to $V$.

\begin{proposition}\label{prop:geometric kernel def}
\begin{itemize}
\item[$(a)$]The series in formula \eqref{geometric aut kernel} converges normally with respect to the operator norm in the ring $\mathrm{End}(V)$ of endomorphisms of $V$ in the variables $(z,w;s)$, with $s$ in the half-plane $\Re(s)>1$ and $z,w\in \mathcal{F}$.
It defines a holomorphic function of $s$ in the half-plane $\Re(s)>1$. The convergence is uniform when $z,w\in \mathcal{F}$ are restricted to any compact subset of $\mathcal{F}$.

\item[$(b)$] The kernel $\tilde K_s(z,w)$ is a meromorphic function of $s$ in the half-plane $\Re(s)>1$, possessing simple poles in this half-plane only when $|k|>1$. When $|k|>1$, the simple poles are located at $s=|k|-n$, for integers $n\in[0,|k|-1]$.

\item[$(c)$]
For each $w\in \HH$ and $s\in\CC$ with $\Re(s)>\max\{1,|k|\}$ each column of the matrix $\tilde K_{s}(\cdot,w)$ defines a function in $\mathcal{H}_k$.
\end{itemize}
\end{proposition}
\begin{proof}
$(a)$ For any $\gamma\in\widetilde{\Gamma}$, the hypergeometric function $F(-k,k;s;(1+\cosh(d_{\hyp}(z,\gamma w)))^{-1})$ is well-defined for all $s\in \mathbb{C}$, due to the fact that
$$0<(1+\cosh(d_{\hyp}(z,w)))^{-1}\leq \tfrac{1}{2}$$
for any two points $z,w\in \mathcal F$
(equality being attained when $z=w$). Moreover, for all $s$ with $\Re(s)>1$, the function $F(-k,k;s;(1+\cosh(d_{\hyp}(z,\gamma w)))^{-1})$ is holomorphic, since it is the sum of uniformly convergent holomorphic functions. For all $s\in\CC$ such that $\Re(s)>1$, it is uniformly bounded by $F\left(-k,k;1,\tfrac{1}{2}\right)$.

Since $\chi$ is a unitary multiplier system and $|J_{\gamma,k}(w)|=|H_k(z,w)|=1$, when $d=1$, the series appearing in the definition of the kernel $\tilde K_s(z,w)$ is dominated (uniformly in $s$, for $\Re(s)>1$) by the series
\begin{equation}\label{dominating sum}
\sum_{\gamma\in\widetilde{\Gamma}}\cosh(d_{\hyp}(z,\gamma w))^{-\Re(s)},
\end{equation}
which converges in the half-plane $\Re(s)>1$ (see e.g. \cite[Lemma 3.3.4]{vP10}). The convergence is uniform when $z,w\in \mathcal F$ are restricted to any compact subset of $\mathcal F$.

When $d>1$, in order to prove the normal convergence of the series in \eqref{geometric aut kernel}, it is sufficient to notice that $\chi$ can be identified with a unitary $d\times d$ matrix, with matrix norm induced from the Hilbert space norm obviously equal to $\sqrt{d}$. The normal convergence follows again from the convergence of the series \eqref{dominating sum} for $\Re(s)>1$, which is uniform when $z,w\in \mathcal F$ are restricted to any compact subset of $\mathcal F$. This proves part $(a)$ of the Proposition.

$(b)$ From part $(a)$ it follows that the sum over $\widetilde{\Gamma}$ on the right-hand side of \eqref{geometric aut kernel} is a holomorphic function in $s$, for $\mathrm{Re}(s)>1$. Therefore the poles of $\tilde K_s(z,w)$ in the half-plane $\mathrm{Re}(s)>1$ stem only from possible poles of the factor $\Gamma(s-k)\Gamma(s+k)/\Gamma(s)^2$. This factor can have poles in the half-plane $\mathrm{Re}(s)>1$ only when $|k|>1$, and they are located at $s=|k|-n$, for integers $n\in[0,|k|-1]$.

(c) To prove the last part, use \cite[formula (6.11) on p. 387]{He83} and note that $L_2(\Gamma\backslash\HH, m, \mathcal{W})=\mathcal{H}_k$ in Hejhal's notation. Therefore, it suffices to show that the function \begin{equation}\label{phi formula}
\Phi_s(4u):=\sqrt{\frac{2}{\pi}}\cdot\frac{\Gamma(s-k)\Gamma(s+k)}{\Gamma(s)^2}(1+2u)^{-s}F(-k,k;s;\tfrac{1}{2(1+u)}),
\end{equation}
where $u=u(z,w)$ is given by \eqref{def_u}, satisfies \cite[Assumption 6.1 on p. 387]{He83}.

It is clear from the definition of
$F(-k,k;s;\tfrac{1}{2(1+u)})$ that it is four times differentiable in $u\geq 0$ and its derivatives are uniformly bounded by $F\left(-k,k;1,\tfrac{1}{2}\right)$. 
The function $F_s(u)=(1+2u)^{-s}$ is also four times differentiable as a function of $u$ for any fixed $s\in\CC$ with $\Re(s)>\max\{1,|k|\}$ and satisfies the bound
$$\left|F_s^{(j)}(u)\right| \ll (1+u)^{-j-\Re(s)},$$ for $j=0,1,2,3,4$, where the implied constant is independent of $u$. It follows that $\Phi_s(t)$ is four times differentiable as a function of the real parameter $t\geq 0$. Furthermore, for every $s\in\CC$ such that $\Re(s)>\max\{1,|k|\}$, the estimate
$$\left|\Phi_s^{(j)}(t)\right| \ll (4+t)^{-j-\Re(s)}$$ holds for $j=0,1,2,3,4$ and $t\geq 0$. Thus, the proof is complete.
\end{proof}

According to Proposition \ref{prop:geometric kernel def}, for any fixed complex number $s$ with $\Re(s)>\max\{1,|k|\}$, the kernel $\tilde{K}_s(z,w)$ can be viewed as a map from $\mathcal F\times \mathcal F$ to $\mathrm{End}(V)$.
In the following proposition we prove that for all such $s$, automorphic kernels $\tilde K_s(z,w)$ and $K_s(z,w)$ are equal on $\mathcal F\times \mathcal F$.

\begin{proposition}
For all $s\in\mathbb{C}$ with $\Re(s)>\max\{1,|k|\}$ and for all $z, w\in\mathcal F$,
$$ \tilde K_s(z,w)\:=\: K_s(z,w).$$
\end{proposition}
\begin{proof}
In view of Proposition \ref{prop:geometric kernel def}.$(c)$, it suffices to show that, for $\Re(s)>\max\{1,|k|\}$, the functions $\tilde K_s(z,w)$ and $K_s(z,w)$ have the same coefficients in the spectral expansion, which amounts to showing that the function $\Phi_s(4u)$ defined in \eqref{phi formula} is the inverse Selberg Harish-Chandra transform of the Fourier transform $h_s$ of the function
\begin{equation}\label{gs function}
g_s(u)=\frac{\Gamma(s-\tfrac{1}{2})}{\Gamma(s)}(\cosh u)^{-(s-\tfrac{1}{2})}=Q_s(e^u+e^{-u}-2).
\end{equation}
Based on \cite[Formula (6.6) on p. 386]{He83}, this is equivalent to showing that
\begin{equation}\label{phi integral}
\Phi_s(4u)=\frac{\Gamma(s+\tfrac{1}{2})}{\pi\Gamma(s)}2^{s-\tfrac{1}{2}}\int\limits_{-\infty}^{\infty}(4u+t^2+2)^{-(s+\tfrac{1}{2})}
\left(\frac{\sqrt{4u+4+t^2}-t}{\sqrt{4u+4+t^2}+t}\right)^k\, dt.
\end{equation}

Substitute $y=4u$ and denote by $I(y)$ the integral on the right-hand side of \eqref{phi integral}. It follows immediately that
$$
I(y)=\int\limits_{0}^{\infty}(\alpha+t^2-2)^{-(s+\tfrac{1}{2})}\left[
\left(\frac{\sqrt{\alpha+t^2}-t}{\sqrt{\alpha+t^2}+t}\right)^k +\left(\frac{\sqrt{\alpha+t^2}-t}{\sqrt{\alpha+t^2}+t}\right)^{-k}\right]\, dt,
$$
where $\alpha=y+4$. Introducing a new variable $x=\frac{\sqrt{\alpha+t^2}-t}{\sqrt{\alpha+t^2}+t}$, we obtain that
$$
I(\alpha-4)=\alpha^{-s}4^{s-\tfrac{1}{2}}\int\limits_0^1(x^2+2x(1-\tfrac{4}{\alpha})+1)^{-(s+\tfrac{1}{2})}((x^{s+k}+x^{s-k-1})+(x^{s-k}+x^{s+k-1}))dx.
$$
Substituting $x_1=1/x$ in the two integrals containing exponents $x^{s-k-1}$ and $x^{s+k-1}$ yields the following simplified expression:
$$
I(y)=\alpha^{-s}4^{s-\tfrac{1}{2}}\int\limits_0^\infty\frac{x^{s+k} + x^{s-k}}{(x^2+2x\tfrac{y}{y+4}+1)^{s+\tfrac{1}{2}}}dx.
$$

Next, write $I(y)=I(4u)=I_+(4u)+I_-(4u)$, where
$$
I_\pm(4u)= \frac{1}{2}(u+1)^{-s}\int\limits_0^\infty\frac{x^{s\pm k} }{(x^2+2x\tfrac{u}{u+1}+1)^{s+\tfrac{1}{2}}}dx.
$$
The integral on the right-hand side of the above equation appears in Formula 8.714.2 of \cite{GR07} for the integral representation of the Legendre function, with $\cos(\varphi)=u/(u+1)\in(0,1)$, $\mu=s$ and $\nu=\pm k$. For $\Re(s)>\max\{1,|k|\}$, the conditions $\Re(\mu\pm\nu)>0$ are fulfilled, hence
\begin{align*}
I_\pm(4u)&=\frac{1}{2}(u+1)^{-s}\int\limits_0^\infty\frac{x^{s\pm k} }{(x^2+2x\tfrac{u}{u+1}+1)^{s+\tfrac{1}{2}}}dx \\
&= \frac{2^{s-1}\Gamma(s+1)\Gamma(s\pm k +1)\Gamma(s\mp k)}{\Gamma(2s+1)(1+2u)^{s/2}}P_{\pm k}^{-s}\left(\frac{u}{u+1}\right),
\end{align*}
where $P_\nu^\mu$ stands for the Legendre function. Inserting the above expression for $I_\pm(4u)$ into \eqref{phi integral} after applying the doubling formula $\Gamma(2s+1)=2^{2s}\pi^{-\tfrac{1}{2}}\Gamma(s+\tfrac{1}{2})\Gamma(s+1)$ for the gamma function, we obtain that
$$
\Phi_s(4u)=\frac{2^{-\tfrac{3}{2}}}{\sqrt{\pi}}\frac{\Gamma(s+k)\Gamma(s-k)}{\Gamma(s)}\left((s+k)P_k^{-s}\left(\frac{u}{u+1}\right)
+ (s-k)P_{-k}^{-s}\left(\frac{u}{u+1}\right)\right)(1+2u)^{-\tfrac{s}{2}}.
$$
In order to prove \eqref{phi formula}, it is left to show that
$$
(s+k)P_k^{-s}\left(\frac{u}{u+1}\right)+(s-k)P_{-k}^{-s}\left(\frac{u}{u+1}\right)=\frac{2}{\Gamma(s)}
(1+2u)^{-\tfrac{s}{2}} F\left(-k,k;s;\tfrac{1}{2(1+u)}\right).
$$
Apply \cite[Formula 8.704]{GR07}, with $x=u/(u+1)\in(0,1)$, $\mu=-s$ and $\nu=\pm k$, in order to express the Legendre function in terms of the hypergeometric function:
$$
P_{\pm k}^{-s}\left(\frac{u}{u+1}\right)=\frac{(1+2u)^{-\tfrac{s}{2}}}{s\Gamma(s)}F\left(\mp k,\pm k +1; s+1;\tfrac{1}{2(1+u)}\right).
$$
Therefore, proof of \eqref{phi formula} reduces to proving that
\begin{equation*}
\begin{split}
&\frac{1}{s}\left((s+k)F\left(- k, k+1; s+1;\tfrac{1}{2(1+u)} \right) + (s-k)F\left( k, -k+1; s+1;\tfrac{1}{2(1+u)} \right)\right)\\
&= 2F\left(-k,k;s;\tfrac{1}{2(1+u)}\right).
\end{split}
\end{equation*}
The above identity follows immediately from the definition of the hypergeometric function and the property $(a+1)_j=(a)_j\tfrac{a+j}{a}$ of the Pochammer symbol $(a)_j=\Gamma(a+j)/\Gamma(a)$, for all non-negative integers $j$, applied with $a=k$ and $a=-k$.
\end{proof}

\begin{remark}
When $k=0$ and $\chi$ is the identity, the hypergeometric series $F\left(-k,k;s;\tfrac{1}{2(1+u)}\right)$ is identically equal to one, hence the series $\tilde K_s(z,w)$ coincides with the automorphic kernel $K_s(z,w)$ defined in \cite[Formula (19)]{JvPS16}, up to the constant $\frac{1}{\sqrt{2\pi}}$.
\end{remark}

By the uniqueness of meromorphic continuation, combining the above proposition with Theorem \ref{merom.cont.}, we arrive at the following corollary:

\begin{corollary} \label{cor:merom cont}
For any $z,w \in \mathcal F$, the geometric kernel $\tilde{K}_s(z,w)$ admits a meromorphic continuation to the whole complex $s$-plane.  The possible poles of the function $\G(s)\G(s-1/2)^{-1}\tilde{K}_s(z,w)$ are located at the points $s=1/2\pm it_j-2n$, where $n\in\NN$ and $\lambda_j=1/4+t_j^2$ is a discrete eigenvalue of $\Delta_k$. In case when $M$ is non-compact, possible poles of $\tilde{K}_s(z,w)$ are also located at the points $s=1-\rho-2n$, where $n\in\NN$ and $\rho\in(1/2,1]$ is a pole of the parabolic Eisenstein series $E_{jl}(z,s)$, and at the points $s=\rho-2n$, where $n\in\NN$ and $\rho$ is a pole of $E_{jl}(z,s)$ with $\Re(\rho)<1/2$.
\end{corollary}

\begin{remark}\rm
Corollary \ref{cor:merom cont} illustrates the strength of the approach to constructing Poincar\'e series using generating kernels. Namely, in order to deduce the meromorphic continuation of the automorphic kernel $\tilde{K}_s(z,w)$ from its geometric definition \eqref{geometric aut kernel},  one would have to consider some type of Fourier expansion (e.g. an expansion in rectangular or spherical coordinates at a certain point) and investigate certain properties of the coefficients in the expansion (e.g. uniform boundedness, analyticity, etc.). This is a heavy task, which we overcome by considering the wave distribution acting on the function $g_s$ defined in \eqref{gs function}.
\end{remark}

\end{document}